\pdfoutput=1 % needed for arXiV submission if use microtype package

\documentclass[12pt]{extarticle}
\usepackage[utf8]{inputenc}
\usepackage[T1]{fontenc}

\usepackage{caption}
\oddsidemargin \evensidemargin
\usepackage{float}
\usepackage{resizegather}
\usepackage{amssymb}
\usepackage{pmboxdraw}
\usepackage{amsmath}
\usepackage{bbm}
\usepackage{amsthm}
\usepackage{amssymb}
\usepackage{dsfont}
\usepackage[dvipsnames]{xcolor}
\usepackage{graphicx}
\usepackage{alphabeta}
\usepackage{tikz,pgfplots}
\usepackage{tikz-cd}
\pgfplotsset{compat=1.18}
\usepackage{pgfplots}
\usepackage{comment}
\usepackage[margin=0.5cm]{caption}
\usepackage{subcaption}
\usepackage[all]{xy}
\usepackage[greek,english]{babel}
\usepackage[export]{adjustbox}
\usepackage{hyperref}
\usepackage{blkarray} 
\usepackage{mathrsfs}  
\hypersetup{
    colorlinks=true,
    linkcolor=myblue,
    filecolor=magenta,      
    urlcolor=myblue,
    citecolor=BurntOrange,
}
\usepackage{enumitem}
\usepackage{booktabs}

\definecolor{myblue}{rgb}{0.25,0.45,0.99}
\definecolor{myorange}{rgb}{0.8500, 0.3250, 0.0980}
\definecolor{myyellow}{rgb}{0.9290, 0.6940, 0.1250}
\definecolor{mypurple}{rgb}{0.4940, 0.1840, 0.5560}
\definecolor{mygreen}{rgb}{0.4660, 0.6740, 0.1880}

\usepackage[activate={true,nocompatibility},final,tracking=true,kerning=true,spacing=true,factor=1100,stretch=10,shrink=10]{microtype}

\let\OLDthebibliography\thebibliography
\renewcommand\thebibliography[1]{
  \OLDthebibliography{#1}
  \setlength{\parskip}{0.4pt}
  \setlength{\itemsep}{3.0pt plus 0.3ex}
}

\allowdisplaybreaks

\usepackage{listings}

\newtheorem{theorem}{Theorem}[section]
\newtheorem{algorithm}{Algorithm}
\newtheorem{theorem*}[theorem]{Theorem*}

\newtheorem{lemma}[theorem]{Lemma}
\newtheorem{corollary}[theorem]{Corollary}
\newtheorem{proposition}[theorem]{Proposition}
\newtheorem{assumption}{Assumption}
\newtheorem{thm/conj}[theorem]{Theorem/Conjecture}

\newtheorem{conjecture}[theorem]{Conjecture}

\theoremstyle{definition}
\newtheorem{examplex}[theorem]{Example}
\newenvironment{example}
{\pushQED{\qed}\examplex}
{\popQED\endexamplex}

\newtheorem{definition}[theorem]{Definition}

\theoremstyle{remark}

\renewcommand{\P}{\mathbb{P}}
\newcommand{\C}{\mathbb{C}}
\newcommand{\R}{\mathbb{R}}

\DeclareMathOperator{\Newt}{Newt}
\DeclareMathOperator{\Nef}{Nef}
\DeclareMathOperator{\Cone}{Cone}

\newcommand\restr[2]{{
  \left.\kern-\nulldelimiterspace 
  #1
  \vphantom{\big|} 
  \right|_{#2}
  }}

\usepackage{titling} % decrease space above title
\usepackage{authblk} % For affiliations
\title{\textbf{Positive Charts of Toric Varieties}
}
\author[1]{Veronica Calvo Cortes}
\author[2]{Simon Telen}
\affil[1]{\small MPI MiS, Inselstrasse 22, 04103 Leipzig, Germany, \href{mailto:veronica.calvo@mis.mpg.de}{veronica.calvo@mis.mpg.de}}
\affil[2]{MPI MiS, Inselstrasse 22, 04103 Leipzig, Germany, \href{mailto:simon.telen@mis.mpg.de}{simon.telen@mis.mpg.de}}

\date{}

\begin{document}

\maketitle

\vspace{-1cm}
\begin{abstract}
\noindent  We construct affine charts of a smooth projective toric variety which contain its nonnegative points, and which admit a closed embedding into the total coordinate space of Cox's quotient construction. We show that such positive charts arise from smooth subcones of the nef cone. To each positive chart we associate an algebraic moment map, the fibers of which are the critical points of a monomial function in Cox coordinates. This work provides a toric framework for the theory of $u$-equations in positive geometry.
\end{abstract}


\normalsize
\section{Introduction}

The nonnegative part of the $d$-dimensional complex projective space $\mathbb{P}^d$ is given by 
\[ \mathbb{P}^d_{\geq 0} \, = \, \{(u_0: \cdots: u_d) \in \mathbb{P}^d \, : \, u \in \mathbb{R}^{d+1}_{\geq 0}\setminus \{ 0 \} \}. \]
This is a real manifold with corners, isomorphic to a $d$-simplex. None of the standard affine charts $u_i \neq 0$ contains $\mathbb{P}^d_{\geq 0}$. A convenient alternative is the chart $\mathcal{U} = \mathbb{P}^d \setminus V(u_0 + \cdots + u_d) \supset \mathbb{P}^d_{\geq 0}$. The restriction of the quotient representation $\pi: \mathbb{C}^{d+1} \setminus \{0 \} \rightarrow \mathbb{P}^d$ to the hyperplane \[ U \, = \, \{ u \in \mathbb{C}^{d+1} \setminus \{0\} \, : \, u_0 + \cdots + u_d - 1 = 0 \} \]
uniquely picks a set of homogeneous coordinates $\pi^{-1}(p) \cap U$ for each $p \in \mathcal{U}$. In particular, the semi-algebraic subset of $U$ given by $u_i \geq 0$ is identified with $\mathbb{P}^d_{\geq 0}$. Note that $U$ is the image of
\begin{equation} \label{eq:paramYPd} (t_1, \ldots, t_d) \, \longmapsto \, \Big (\frac{t_1}{1 + t_1 + \cdots + t_d}, \,  \ldots, \, \frac{t_d}{1+ t_1 + \cdots + t_d}, \,  \frac{1}{1 + t_1 + \cdots + t_d} \Big ) .\end{equation}
Hence, it admits a positive rational parametrization whose Newton polytope is isomorphic to $\mathbb{P}^d_{\geq 0}$ as a real manifold with corners. Here ``positive'' means that all numerators and denominators have nonnegative coefficients, and the Newton polytope of a rational parametrization is the Minkowski sum of the Newton polytopes of all these polynomials. That is,
\[ \varphi \colon t \, \mapsto \,  \Big (\frac{r_1(t)}{s_1(t)}, \ldots, \frac{r_n(t)}{s_n(t)} \Big)  \; \text{ has Newton polytope } \; {\rm Newt}(\varphi) = \sum_{i = 1}^n ({\rm Newt}(r_i) + {\rm Newt}(s_i))\]
where $r_i$ and $s_i$ are coprime. Our paper extends these constructions to any smooth projective toric variety, providing a positive rational parametrization for its nonnegative part. 

Let $\Sigma$ be a smooth rational fan in $\mathbb{R}^d$. The smooth toric variety $X_\Sigma$ is the GIT quotient
\begin{equation} \label{eq:quotientrep} 
X_\Sigma \, = \, (\mathbb{C}^{\Sigma(1)} \setminus Z(\Sigma)) / G\end{equation}
where $\Sigma(1)$ is the set of rays of $\Sigma$, $Z(\Sigma)$ is a union of coordinate subspaces and the quotient is by the action of the torus $G = {\rm Hom}_\mathbb{Z}({\rm Cl}(X_\Sigma), \mathbb{C}^*)$ \cite{Cox1995}. %In particular, $G$-orbits in $\mathbb{C}^{\Sigma(1)} \setminus Z(\Sigma)$ are in one-to-one correspondence with points in $X_\Sigma$. 
A point $u = (u_\rho)_{\rho \in \Sigma(1)} \in \mathbb{C}^{\Sigma(1)} \setminus Z(\Sigma)$ gives homogeneous coordinates or Cox coordinates for its image under the map
\begin{equation} \label{eq:coxquotient}
\pi: \mathbb{C}^{\Sigma(1)} \setminus Z(\Sigma) \, \rightarrow \, X_\Sigma.
\end{equation}
This is the quotient morphism associated to \eqref{eq:quotientrep}.
For instance, if $X_\Sigma = \mathbb{P}^d$, then $Z(\Sigma) = \{0\}$, $G = \mathbb{C}^*$ and $\pi^{-1}(p) \subset \mathbb{C}^{d+1} \setminus \{0\}$ is the $\mathbb{C}^*$-orbit of homogeneous coordinate choices for~$p$.

The nonnegative part of $X_\Sigma$ consists of all points with nonnegative Cox coordinates: 
\[ (X_\Sigma)_{\geq 0} \, = \, \{ p \in X_\Sigma \, : \, \pi^{-1}(p) \cap \mathbb{R}^{\Sigma(1)}_{\geq 0} \neq \emptyset \}.\]
Let $\Sigma = \Sigma_P$ be the normal fan of a smooth polytope $P$. We have that $(X_\Sigma)_{\geq 0}$ is isomorphic to $P$ as a real manifold with corners \cite[\S 4.2]{Fulton1993}. %A very ample dilation $k \cdot P$ of $P$ gives an embedding $\iota: X_\Sigma \hookrightarrow \mathbb{P}^N$, where $N+1$ is the number of lattice points in $k \cdot P$. The nonnegative part $(X_\Sigma)_{\geq 0}$ is the intersection of $\iota(X_\Sigma)$ with $\mathbb{P}^N_{\geq 0}$. 
One of our main results is a constructive proof of the following theorem, which generalizes the above discussion for $X_\Sigma = \mathbb{P}^d$. 

\begin{theorem} \label{thm:mainintro}
    Let $\Sigma$ be the normal fan of a smooth $d$-dimensional lattice polytope $P \subset \mathbb{R}^d$. There exists an affine variety $U\subset \mathbb{C}^{\Sigma(1)} \setminus Z(\Sigma)$ and an affine open subset $\mathcal{U} \subset X_\Sigma$ such~that
    \begin{enumerate}
    \item $\pi_{|U}: U  \overset{\sim}{\longrightarrow} \mathcal{U}$ is an isomorphism, with $\pi$ as in \eqref{eq:coxquotient},
    \item inverting $\pi_{|U}$ yields a positive rational parametrization $\varphi \colon \mathbb{C}^d \dashrightarrow U$ with $\Sigma_{{\rm Newt}(\varphi)} = \Sigma$,
    \item $\pi(U_{\geq 0}) = (X_\Sigma)_{\geq 0}$, where $U_{\geq 0} = U \cap \mathbb{R}^{\Sigma(1)}_{\geq 0}$.
\end{enumerate}
\end{theorem}
Our title proposes a name for affine varieties $U$ with the properties listed in Theorem \ref{thm:mainintro}.
\begin{definition} \label{def:maindef}
    A \emph{positive chart} of the smooth projective toric variety $X_\Sigma$ is an affine variety $U \subset \mathbb{C}^{\Sigma(1)} \setminus Z(\Sigma)$ such that $U$ and $\mathcal{U} = \pi(U)$  satisfy the assertions of Theorem \ref{thm:mainintro}.
\end{definition}

Notice that point 3 in Theorem \ref{thm:mainintro} implies that $(X_\Sigma)_{\geq 0} \subseteq \mathcal{U}$. An interpretation of point 1 is that $U$ picks a unique set of Cox coordinates $\pi^{-1}(p) \cap U$ for each point $p \in \mathcal{U}$. In Section \ref{sec:constructing}, we show how to compute a positive chart for the toric variety of any smooth polytope $P$. In particular, we construct the parametrization from point 2 explicitly, and we shall derive defining equations for $U \cap (\mathbb{C}^*)^{\Sigma(1)}$. The following will serve as our running example. 

\begin{example}\label{ex: pentagon intro}
We consider the normal fan $\Sigma$ of a pentagon in $\R^2$ as in Figure~\ref{fig: pentagon ex}, and we explicitly describe one of its positive charts.
\begin{figure}[h!]
%\begin{subfigure}{0.48\textwidth}
    \centering
    \begin{tikzpicture}[scale=1.1]
  % Vertices of the pentagon
  \coordinate (A) at (0,0);
  \coordinate (B) at (0,2);
  \coordinate (C) at (2,2);
  \coordinate (D) at (2,1);
  \coordinate (E) at (1,0);

    \node at (0,0) {$\bullet$};
    \node at (0,2) {$\bullet$};
    \node at (2,2) {$\bullet$};
    \node at (2,1) {$\bullet$};
    \node at (1,0) {$\bullet$};

  \node[left] at (0,2) {$(0,2)$};
  \node[right] at (2,2) {$(2,2)$};
  \node[right] at (2,1) {$(2,1)$};
  \node[below] at (1,0) {$(1,0)$};

  % Draw pentagon
  \draw[thick] (A)--(B)--(C)--(D)--(E)--cycle;

  % Axes
  \draw[->] (-0.5,0) -- (2.7,0) node[right] {};
  \draw[->] (0,-0.5) -- (0,2.7) node[above] {};
\end{tikzpicture}
\quad \quad \quad 
%\end{subfigure}
%\begin{subfigure}{0.48\textwidth}
%    \centering
    \begin{tikzpicture}[scale=0.85]
  % Axes
  \draw[->] (-2.2,0) -- (2.2,0) node[right] {};
  \draw[->] (0,-2.2) -- (0,2.2) node[above] {};

  % Rays of the fan (column order)
  \draw[thick,->] (0,0) -- (1.6,0);
  \draw[thick,->] (0,0) -- (0,1.6);
  \draw[thick,->] (0,0) -- (-1.6,0);
  \draw[thick,->] (0,0) -- (0,-1.6);
  \draw[thick,->] (0,0) -- (-1.2,1.2);

  % Labels
  \node[above] at (1.6,0) {$1$};
  \node[right] at (0,1.6) {$2$};
  \node[above] at (-1.6,0) {$4$};
  \node[right] at (0,-1.6) {$5$};
  \node[above] at (-1.2,1.2) {$3$};

  % Origin
  \fill (0,0) circle (1.5pt);
\end{tikzpicture}
%\end{subfigure}
    \caption{A pentagon and its normal fan.}
    \label{fig: pentagon ex}
\end{figure}
The base locus $Z(\Sigma)$ is given by the ideal $\langle u_1u_2u_3, u_1u_2u_5, u_3u_4u_5, u_1u_4u_5, u_2u_3u_4\rangle$ and the group $G\cong (\C^*)^3$ acts on $\mathbb{C}^5 \setminus Z(\Sigma)$ as follows 
\[ (\lambda, \mu, \nu) \cdot (u_1,u_2,u_3,u_4,u_5) \, = \, (\lambda \, u_1, \frac{\nu}{\lambda} \, u_2, \frac{\lambda \mu}{\nu} \, u_3, \frac{\nu}{\mu} \, u_4,\mu \, u_5), \quad  \lambda,\mu,\nu \in \C^*.\] 
We recall how to determine $Z(\Sigma)$ and the action of $G$ in Section \ref{sec:prelim}. The affine variety $U \subset \C^5 \setminus Z(\Sigma)$ is a complete intersection cut out by three equations
$$U = V(u_3u_4+u_1-1, \; u_5 + u_2u_3 -1, \;  u_4u_5 + u_2u_3u_4 +u_1u_2 -1 ).$$
It is obtained as the image of the following positive rational parametrization $\varphi:\mathbb{C}^2 \dashrightarrow U$:
\[
 \left(\frac{t_1}{1 + t_1}, \frac{t_2(1 + t_1)}{1 + t_2 +t_1t_2}, \frac{1 + t_2 + t_1t_2}{(1 + t_1)(1+t_2)},\frac{1 + t_2}{1 + t_2 + t_1t_2},\frac{1}{1 + t_2}\right).
\]
The Newton polytope of this map is a pentagon whose normal fan is $\Sigma$: 
\[ {\rm Newt}(\varphi) \, = \, {\rm Newt}(t_1) + {\rm Newt}(t_2) + 3{\rm Newt}(1+t_1) + 3{\rm Newt}(1+t_2) + 3{\rm Newt}(1 + t_2 + t_1t_2).\]
The open set $\mathcal{U}$ is the complement of a union of three irreducible curves in $X_\Sigma$. These are the closures in $X_\Sigma$ of the curves in $(\mathbb{C}^*)^2$ defined by $1 + t_1 = 0, \, 1 + t_2 = 0, \, 1 + t_2 + t_1t_2 = 0$. 
%Consider the embedding of $X_\Sigma$ into $\P^7$ as the closure of the image of the monomial map
%$$(t_1,t_2) \mapsto [1:t_1:t_2:t_1t_2:t_2^2:t_1^2t_2:t_1t_2^2:t_1^2t_2^2].$$ Denoting the homogeneous coordinates of $\P^7$ as $z_0,\ldots,z_7$ we have that the affine open subset $\mathcal{U}$ is given by taking out a hyperplane from the toric variety $X_\Sigma$, explicitly
%$$ \mathcal{U} = X_\Sigma \setminus V(z_0+z_1+2z_2+3z_3+z_4+z_5+2z_6+z_7)$$ 
\end{example}

In Section \ref{sec:moment} we associate a natural moment map to a positive chart $U$. That is, we construct explicit morphisms $\mu_{U,s}: U \to \C^{\Sigma(1)}$ depending on parameters $s$ which identify $U_{\geq 0}$ with a $d$-dimensional polytope $P(s)$ whose normal fan is $\Sigma_{P(s)} = \Sigma$. We describe the fibers $\mu_{U,s}^{-1}(x)$ as the critical points of a multivalued monomial function $\prod_{\rho \in \Sigma(1)} u_\rho^{x_\rho}$ on $U \cap (\mathbb{C}^*)^{\Sigma(1)}$.

\paragraph{Related work.} While not phrased in the language of homogeneous coordinates on toric varieties, the recipe followed in Section \ref{sec:constructing} to construct $U$ and $\mathcal{U}$ first appeared in the physics literature \cite[Sections 9.5 and 10]{ArkaniHamed2021}. In several examples from physics, corresponding to specific choices of the fan $\Sigma$, the variety $U$ is a binary geometry in the sense of \cite{binaryGeometries}, see Section \ref{sec: examples}. In particular, our paper identifies the dihedral coordinates (or $u$-coordinates) on the moduli space $M_{0,n}$ from \cite[Section 2]{BrownMZVModuli} as Cox coordinates on the toric variety associated with a smooth realization of the associahedron. Example \ref{ex: pentagon intro} appears in \cite{ArkaniHamed2021,binaryGeometries,BrownMZVModuli}, and $U$ is shown to be a partial compactification of $M_{0,5}$ in \cite{BrownMZVModuli}.

The recent paper \cite{had} presents a large family of examples of $U$-varieties arising in a representation-theoretic context. There, the fan $\Sigma$ is the $g$-vector fan of a finite representation type $\mathbb{C}$-algebra, and such an algebra serves as the starting point of the construction of $U$. In contrast, our point of departure is the fan $\Sigma$ itself. We prove Theorem~\ref{thm:mainintro} using only standard techniques from toric geometry; no representation theory enters the argument. It would be interesting to further explore the relationship between these toric and representation-theoretic perspectives. In particular, the homogenization procedure in \cite[Sections~5 and~9]{had} closely parallels our homogenization with respect to the Cox ring of $X_\Sigma$.

\paragraph{Outline. } Section \ref{sec:prelim} recalls some facts from toric geometry and fixes our notation. Section \ref{sec:constructing} contains our main construction of a positive chart for any smooth fan. In particular, it contains a proof of Theorem \ref{thm:mainintro} and Algorithm \ref{algo} describes how to compute a positive chart $U$ from $\Sigma$. In Section \ref{sec: examples} we work out some relevant examples, and we illustrate our Julia code which implements Algorithm \ref{algo} \cite{zenodo}. Section \ref{sec:moment} discusses moment maps of positive charts.

\section{Toric preliminaries} \label{sec:prelim}

This section recalls some concepts from toric and polyhedral geometry, and it fixes our notation. More background and details can be found in the texts \cite{Fulton1993,CoxLittleSchenckToric,Telen2025AppliedToricGeometry}. 

We start from a complete simplicial rational polyhedral fan $\Sigma$ in $\mathbb{R}^d$. The set of $\ell$-dimensional cones of $\Sigma$ is $\Sigma(\ell)$. The fact that $\Sigma$ is \emph{simplicial} means that each $\sigma \in \Sigma(d)$ has $d$ rays. Our fan $\Sigma$ has $n$ rays in total: $\Sigma(1) = \{\rho_1, \ldots, \rho_n\}$. By abuse of notation, we also write $\rho_i \in \mathbb{Z}^d$ for the primitive generator of the ray $\rho_i$, and we collect these vectors in a matrix $F = \begin{pmatrix}
    \rho_1 & \rho_2 & \cdots & \rho_n
\end{pmatrix} \in \mathbb{Z}^{d \times n}$.
We are exclusively interested in the case where $\Sigma = \Sigma_P$ is the normal fan of a $d$-dimensional polytope $P \subset \mathbb{R}^d$. Since $\Sigma$ is simplicial, $P$ is simple, meaning that each of its vertices is contained in precisely $d$ facets. The fact that $\Sigma = \Sigma_P$ implies that $P$ can be written as 
\[ P \, = \, \{ m \in \mathbb{R}^d \, : \, F^t m + a_P \geq 0 \}\]
for a unique integer vector $a_P = (a_{P,1}, \ldots, a_{P,n}) \in \mathbb{Z}^n$. By the inequality $F^tm + a_P \geq 0$ we mean the entry-wise inequality: $\langle \rho_i, m \rangle + a_{P,i} \geq 0, \, i = 1, \ldots, n$. The \emph{type cone} ${\rm Type}(\Sigma)$ of $\Sigma$ is an open convex cone in $\mathbb{R}^n$ consisting of all points $z \in \mathbb{R}^n$ such that the normal fan of 
\[ P_z = \{ m \in \mathbb{R}^d \, : \, F^t m + z \geq 0 \} \] 
is $\Sigma$. Note that $a_P \in {\rm Type}(\Sigma)$. The \emph{deformation cone} ${\rm Def}(\Sigma)$ is the closure of ${\rm Type}(\Sigma)$ in the Euclidean topology on $\mathbb{R}^n$. It has a $d$-dimensional lineality space given by ${\rm im}_{\mathbb{R}} F^t$; the $\mathbb{R}$-span of the rows of $F$. Taking the quotient by this lineality space gives the \emph{nef cone} of $\Sigma$: 
\[ {\rm Nef}(\Sigma) \, = \, {\rm Def}(\Sigma)/{\rm im}_{\mathbb{R}} F^t .\]
The terminology \emph{nef cone} comes from divisor theory on toric varieties. Let $X_\Sigma$ be the normal projective toric variety associated to our fan $\Sigma$. A \emph{torus-invariant divisor} on $X_\Sigma$ is a $\mathbb{Z}$-linear combination of the prime divisors $D_1, \ldots, D_n$ associated with the rays $\rho_1, \ldots, \rho_n$ respectively. These divisors form a free group ${\rm Div}_T(X_\Sigma) \simeq \mathbb{Z}^n$, where the isomorphism represents a torus-invariant divisor $z_1 \, D_1 + \cdots + z_n \, D_n \in {\rm Div}_T(X_\Sigma)$ by its coefficients $z = (z_1, \ldots, z_n) \in \mathbb{Z}^n$. Two such divisors $z, z'$ are \emph{linearly equivalent} if $z - z' = F^t m$ for some $m \in \mathbb{Z}^d$. The divisor class group is the quotient of ${\rm Div}_T(X_\Sigma) \simeq \mathbb{Z}^n$ by linear equivalence: 
\[ {\rm Cl}(X_\Sigma) \,  \simeq \, \mathbb{Z}^n / {\rm im}_{\mathbb{Z}} F^t. \]
A Cartier divisor $z_1 \, D_1 + \cdots + z_n \, D_n \in {\rm Div}_T(X_\Sigma)$ is \emph{numerically effective}, or \emph{nef} for short, if and only if $z \in {\rm Def}(\Sigma) \cap \mathbb{Z}^n$. It is \emph{ample} if and only if $z \in {\rm Type}(\Sigma) \cap \mathbb{Z}^n$, that is, if and only if $\Sigma_{P_z} = \Sigma$. If ${\rm Cl}(X_\Sigma)$ has no torsion, which is an assumption we will need later on (see Example \ref{ex:torsion}), then ${\rm Cl}(X_\Sigma) \simeq \mathbb{Z}^{n-d}$ is a lattice inside $\mathbb{R}^n/ {\rm im}_{\mathbb{R}}F^t \simeq \mathbb{R}^{n-d}$, and ${\rm Nef}(\Sigma)$ is a pointed rational polyhedral cone inside that same vector space. The Picard group ${\rm Pic}(X_\Sigma) \subseteq {\rm Cl}(X_\Sigma)$ is a sublattice of finite index, and ${\rm Nef}(\Sigma) \cap {\rm Pic}(X_\Sigma)$ consists of the nef divisor classes on $X_\Sigma$. Similarly, the ample Cartier divisor classes are the points in ${\rm int}({\rm Nef}(\Sigma)) \cap {\rm Pic}(X_\Sigma)$. We shall mostly assume that $\Sigma$ is a \emph{smooth} fan, which means that for each $\sigma \in \Sigma(d)$ the primitive ray generators of $\sigma$ form a $\mathbb{Z}$-basis for $\mathbb{Z}^d$. This holds if and only if we have the equality ${\rm Cl}(X_\Sigma) = {\rm Pic}(X_\Sigma)$. The next examples illustrate these concepts. 

\begin{example}\label{ex: pentagon 2}
    The ray matrix of the fan $\Sigma$ in Example~\ref{ex: pentagon intro} is $F = \left ( \begin{smallmatrix}
        1 & 0 & -1 & -1 & 0\\ 0 &1 & 1 & 0 & -1 
    \end{smallmatrix} \right)$.
    The class group is generated by $[D_1], \ldots, [D_5]$, modulo the following relations read from the rows of $F$: $[D_1]-[D_3]-[D_4]=0$ and $[D_2]+[D_3]-[D_5]=0$. We deduce that \[ [a_1D_1+a_2D_2+a_3D_3+a_4D_4+a_5D_5]=(a_1+a_3-a_2)[D_3]+(a_1+a_4)[D_4]+(a_2+a_5)[D_5] \] 
    so that $[D_3], [D_4], [D_5]$ generate ${\rm Cl}(X_\Sigma) \simeq \mathbb{Z}^3$ freely. Since $\Sigma$ is smooth we have ${\rm Cl}(X_\Sigma) = {\rm Pic}(X_\Sigma)$. The nef divisor classes are those torus-invariant divisor classes $a_3[D_3]+ a_4[D_4]+a_5[D_5]$ for which $(a_3,a_4,a_5)$ lies in the nef cone of $\Sigma$. This is the smooth rational polyhedral cone ${\rm Nef}(\Sigma) \subset \R^3$ with ray generators $(1,1,0),(0,0,1)$ and $(0,1,1)$. We point out that, in general, even for smooth $\Sigma$, the nef cone ${\rm Nef}(\Sigma) \subset \mathbb{R}^{n-d}$ need not be smooth or simplicial.
\end{example}

\begin{example} \label{ex:P121}
Let $\Sigma$ be the normal fan of a triangle with vertices $(0,0),(0,1)$ and $(2,0)$. It has ray matrix $F = \left ( \begin{smallmatrix}
    1 & 0 & -1 \\ 0 & 1 & -2
\end{smallmatrix} \right)$. The associated toric surface is a weighted projective plane $X_\Sigma=\P(1,2,1)$. The class group ${\rm Cl}(X_\Sigma) \simeq \mathbb{Z}$ is generated by $[D_3]$ since $[a_1D_1+a_2D_2+a_3D_3] = [(a_1+2a_2+a_3)D_3]$. The divisor $D_3$ is \emph{not} a Cartier divisor but $2D_3$ is. Hence, we get ${\rm Pic}(X_\Sigma)\cong 2\mathbb{Z} \subset \mathbb{Z} \cong {\rm Cl}(X_\Sigma)$. The nef cone is the positive real line. We can think of the nef divisors as even lattice points in ${\rm Nef}(\Sigma)$, i.e., nonnegative even multiples of $D_3$. The nonnegative odd multiples of $D_3$ correspond to non-Cartier divisors whose polytope has normal fan $\Sigma$. For instance, the polytope of $D_3 = 0 \, D_1 + 0 \, D_2 + 1 \, D_3$ is \emph{not} a lattice polytope:
\begin{equation} \label{eq:P001} P_{(0,0,1)} \, = \, \{ m \in \mathbb{R}^2 \, : \, m_1 \geq 0, \, m_2 \geq 0, \, -m_1-2m_2 + 1 \geq 0\}. \end{equation}
It is the triangle with vertices $(0,0),(1,0)$ and $(0,1/2)$. This confirms that $D_3$ is not Cartier. However, we have $\Sigma_{P_{(0,0,1)}} = \Sigma$, which explains that $1 \in {\rm Nef}(\Sigma)$.
\end{example}

For any simplicial fan $\Sigma$ in $\mathbb{R}^d$ whose ray matrix has rank $d$, the toric variety $X_\Sigma$ can be realized as a GIT quotient \cite{Cox1995} of a quasi-affine variety $\mathbb{C}^n \setminus Z(\Sigma)$ by the action of a reductive group $G$: $X_\Sigma \simeq (\mathbb{C}^n \setminus Z(\Sigma))/G$. The affine variety $Z(\Sigma)$ is a union of coordinate subspaces, defined by the Stanley-Reisner ideal of the Alexander dual of $\Sigma$. That is, $Z(\Sigma)$ is defined by the monomial ideal $B(\Sigma) = \langle u^{\hat{\sigma}} \, : \, \sigma \in \Sigma \rangle \subset \mathbb{C}[u_1, \ldots, u_n]$, where $u^{\hat{\sigma}}$ is the product of all $u_i$ such that $\rho_i \not \subset \sigma$. The ideal $B(\Sigma)$ is also called the \emph{irrelevant ideal} of $\Sigma$. It remains to define the group $G$ and its action on $\mathbb{C}^n \setminus Z(\Sigma)$. If ${\rm Cl}(X_\Sigma)$ is torsion-free, then $G$ is a torus of dimension $k = n-d$. Let $K \in \mathbb{Z}^{n \times k}$ be a matrix representing $\ker_{\mathbb{Z}} F$, i.e., ${\rm im}_{\mathbb{Z}} \, K = \ker_{\mathbb{Z}}F$. Its rows are denoted by $v_1, \ldots, v_n \in \mathbb{Z}^k$. A point $(\lambda_1, \ldots, \lambda_k) \in (\mathbb{C}^*)^k = G$ acts as follows:
\[ (\lambda_1, \ldots, \lambda_k) \cdot (u_1, \ldots, u_n) \,= \, (\lambda^{v_1} \, u_1, \ldots, \lambda^{v_n} \, u_n), \quad u \in \mathbb{C}^n \setminus Z(\Sigma).\]
\begin{example} \label{ex:Pdcox}
    Let $\Sigma$ be the normal fan of the standard simplex $P \subset \mathbb{R}^d$. Its ray matrix is $F = \begin{pmatrix}
        {\rm id}_{d \times d} & - {\bf 1}
    \end{pmatrix} \in \mathbb{Z}^{d \times (d+1)}$, where ${\bf 1} \in \mathbb{Z}^d$ is the all-ones vector. The associated toric variety is $X_\Sigma = \mathbb{P}^d$. In this case, Cox's quotient construction is the familiar equality $\mathbb{P}^d = (\mathbb{C}^{d+1} \setminus \{ 0 \})/\mathbb{C}^*$ identifying points in $\mathbb{P}^d$ with lines through the origin in $\mathbb{C}^{d+1}$. The irrelevant ideal is $B(\Sigma) = \langle u_1, \ldots, u_{d+1} \rangle$, so that indeed $Z(\Sigma) = \{0\}$. The matrix $K$ has size $(d+1) \times 1$, and its entries are all equal to $1$. The action of $\mathbb{C}^*$ is therefore given by $\lambda \cdot u = (\lambda u_1, \ldots, \lambda u_{d+1})$.  
\end{example}

\begin{example} \label{ex:P1P1cox}
    Consider the normal fan $\Sigma = \Sigma_P$ of the unit square $P = [0,1]^2 \in \mathbb{R}^2$. Its ray matrix is $F = \left ( \begin{smallmatrix}
        1 & 0 & -1 & 0 \\ 0 & 1 & 0 & -1
    \end{smallmatrix} \right )$. The irrelevant ideal $B(\Sigma) = \langle x_1, x_3 \rangle \cap \langle x_2, x_4 \rangle$ defines a union of two coordinate planes $Z(\Sigma) \subset \mathbb{C}^4$. The group $G = (\mathbb{C}^*)^2$ acts on $\mathbb{C}^4 \setminus Z(\Sigma)$ by 
    \[ (\lambda, \mu) \cdot (u_1,u_2,u_3,u_4) \, = \, (\lambda\, u_1, \mu \, u_2, \lambda\, u_3, \mu \, u_4 ).\]
    The quotient $(\mathbb{C}^4 \setminus Z(\Sigma)) /G$ is a familiar description of $X_\Sigma = \mathbb{P}^1 \times \mathbb{P}^1$.
\end{example}

\begin{example}
    The irrelevant ideal $B(\Sigma)$ and the action of $G$ in our running example were identified in Example \ref{ex: pentagon intro}. Recall the matrix $F$ in that example is $F = \left ( \begin{smallmatrix}
        1 & 0 & -1 & -1 & 0 \\ 0 & 1 & 1 & 0 & -1
    \end{smallmatrix}\right) $.
\end{example}

The quotient morphism $\pi: \mathbb{C}^n \setminus Z(\Sigma) \rightarrow (\mathbb{C}^n \setminus Z(\Sigma))/G =  X_\Sigma$ is a toric morphism induced by the $\mathbb{Z}$-linear map $F: \mathbb{Z}^n \rightarrow \mathbb{Z}^d$. In particular, its restriction $\pi_{|({\mathbb{C}^*})^n}:(\mathbb{C}^*)^n \rightarrow (\mathbb{C}^*)^d \subset X_\Sigma$ is given by the monomial map $\phi_{F^t}: u \mapsto (u^{F_{1,:}}, \ldots, u^{F_{d,:}})$, where $F_{i,:}$ is the $i$-th row of $F$.

We continue to assume that ${\rm Cl}(X_\Sigma)$ is torsion free. We identify ${\rm Cl}(X_\Sigma) \simeq \mathbb{Z}^k$ by setting $[D_i] = v_i$ for each torus-invariant prime divisor. Here $v_i$ is the $i$-th row of the $\mathbb{Z}$-kernel matrix $K$ of $F$. Let $S = \mathbb{C}[u_1, \ldots, u_n]$ be the polynomial ring with one variable for each ray in $\Sigma(1)$. The action of $G$ induces a $\mathbb{Z}^k$-grading on $S$ in which $\deg(u_i) = v_i$. We decompose $S$ as 
\begin{equation} \label{eq:grading} S \, = \, \bigoplus_{w \in \mathbb{Z}^k} S_w, \quad \text{where} \quad S_w \, = \, \bigoplus_{b \in \mathbb{N}^n \, : \, K^t b = w} \mathbb{C} \cdot u^b.\end{equation}
Equivalently, for any $w \in \mathbb{Z}^k$ choose $z \in \mathbb{Z}^n$ such that $K^t z = w$. The monomials in $S_w$ are $u^{F^t m + z}$, where $m$ satisfies $F^t m + z \geq 0$, i.e. $ m \in P_{z} \cap \mathbb{Z}^d$. This gives an identification 
\begin{equation} \label{eq:homogenization} \eta_z \, : \, \bigoplus_{m \in P_{z} \cap \mathbb{Z}^d} \mathbb{C} \cdot t^m \, \overset{\sim}{\longrightarrow} \, S_w \quad \text{given by} \quad \sum_{m \in P_{z} \cap \mathbb{Z}^d} c_m \, t^m \, \longmapsto \, \sum_{m \in P_{z} \cap \mathbb{Z}^d} c_m \, u^{F^tm + z}.\end{equation}
for any choice of $z \in \mathbb{Z}^n$ such that $K^t z = w$. Different choices of $z$ correspond to linearly equivalent divisors. The domain of $\eta_z$ is the space of Laurent polynomials in $\mathbb{C}[t_1^{\pm 1}, \ldots, t_d^{\pm 1}]$ whose Newton polytope is contained in $P_z$. Its codomain is $S_{K^tz}$. The map \eqref{eq:homogenization} is called \emph{homogenization}. Note that one obtains $\eta_z(f)$ by multiplying the pullback $\pi_{|(\mathbb{C}^*)^n}^*(f)$ with $u^z$. That is, usually, one can compute $\eta_z(f)$ by replacing $t = \phi_{F^t}(u)$ and taking the numerator of the resulting rational function. This procedure works for the purposes of our paper. The ring $S$, with its $\mathbb{Z}^k$-grading and its irrelevant ideal $B(\Sigma)$, is called the \emph{Cox ring} of~$X_\Sigma$.

\begin{example}
    For $\Sigma$ from Example \ref{ex:Pdcox}, the grading \eqref{eq:grading} is the standard $\mathbb{Z}$-grading of the polynomial ring. The vector $z = (0, \ldots, 0, w)^t$ with $w \in \mathbb{N}$ satisfies $K^t z = w$. The lattice points of $P_z$ are the monomials in $\mathbb{C}[t_1, \ldots, t_d]$ of degree at most $w$. The homogenization map \eqref{eq:homogenization} is the usual homogenization of a degree $w$ polynomial to a form of degree $w$.
\end{example}

\begin{example} \label{ex:homogenizepentagon}
    The $\mathbb{Z}^3$-grading of $S=\C[u_1,u_2,u_3,u_4,u_5]$ in Example \ref{ex: pentagon intro} is given~by \[ K^t= \begin{pmatrix}
        1 & -1 & 1 & 0 & 0 \\ 
        0 & 0 & 1 & -1 & 1 \\ 
        0 & 1 & -1 & 1 & 0
    \end{pmatrix}, \quad \text{e.g. } \deg(u_3u_4) \, = \, \deg(u_1) \, = \, \begin{pmatrix}
        1\\0\\0
    \end{pmatrix}. \] To homogenize, we replace $t_1=\frac{u_1}{u_3u_4}, \; t_2= \frac{u_2u_3}{u_5}$ and take the numerator. For example,
    \begin{align*}
        1+t_1 &\longmapsto u_3u_4+u_1 \in S_{(1,0,0)}  & 1+t_2+t_1t_2 &\longmapsto u_4u_5 + u_2u_3u_4+u_1u_2  \in S_{(0,0,1)}\\
         1+t_2 &\longmapsto u_5+u_2u_3 \in S_{(0,1,0)}& 1+t_1+t_1t_2 &\longmapsto u_3u_4u_5 + u_1u_5+u_1u_2u_3  \in S_{(1,1,0)}. \qedhere
    \end{align*}
\end{example}

\section{Constructing positive charts} \label{sec:constructing}

Fix $k$ Laurent polynomials $f_1, \ldots, f_k \in \mathbb{C}[t_1^{\pm 1}, \ldots, t_d^{\pm 1}]$. Their Newton polytopes are denoted by $P_i = {\rm Newt}(f_i)$, and we assume that the Minkowski sum $P = P_1 + \cdots + P_k$ is $d$-dimensional. Let $\Sigma = \Sigma_P$ be the normal fan of $P$. As above, we store the primitive ray generators of $\Sigma(1)$ in the columns of a matrix $F = \begin{pmatrix}
    \rho_1 & \rho_2 & \cdots & \rho_n
\end{pmatrix} \in \mathbb{Z}^{d \times n}$.
Here $n = |\Sigma(1)|$ is the number of rays of $\Sigma$. Importantly, we make the assumption that $n = d + k$.

Each of the polytopes $P_i$ corresponds to a basepoint free torus-invariant Cartier divisor \[ D_{P_i} \, = \, \sum_{j = 1}^n a_{i,j} \, D_{j} \, \in \, {\rm Div}_T(X_\Sigma), \quad \text{where} \quad  a_{i, j} = -{\rm min}_{m \in P_i} \langle \rho_j, m \rangle\]
and $D_j$ is the torus-invariant prime divisor corresponding to $\rho_j \in \Sigma(1)$. Consider the matrix
\begin{equation}\label{eq: M matrix}
    M \, = \, \begin{pmatrix}
    \rho_1 & \rho_2 & \cdots & \rho_n \\ 
    a_{1,1} & a_{1,2} & \cdots & a_{1,n} \\
    \vdots & \vdots & \vdots & \vdots \\ 
    a_{k,1} & a_{k,2} & \cdots & a_{k,n} 
\end{pmatrix} \, = \, \begin{pmatrix}
    F \\ 
    a_1^t \\ 
    \vdots \\ 
    a_k^t
\end{pmatrix} \, \in \, \mathbb{Z}^{n \times n}.
\end{equation}
Note that this matrix only depends on the Newton polytopes of the $f_i$: its first $d$ rows are the rays in the normal fan of $P = \sum_i P_i$, and its last $k$ rows are the divisors corresponding to $P_i$ on $X_\Sigma$. 
In what follows, we need to assume that $M$ is unimodular, meaning that $\det M = \pm 1$. 
This implies, for instance, that ${\rm Cl}(X_\Sigma)$ is torsion-free, since the lattice generated by the rows of $F$ is a direct summand of $\mathbb{Z}^n$, and ${\rm Cl}(X_\Sigma) = \mathbb{Z}^n / {\rm im}_{\mathbb{Z}} F^t$. 
Here are two concrete examples. 

\begin{example}\label{ex: pentagon 3}
    We saw in Example~\ref{ex: pentagon 2} that the class group in our running example is torsion-free. We have $k = n-d = 3$ and we consider the polynomials $f_1 = 1 + t_1$, $f_2 = 1 + t_2$ and $f_3 = 1 + t_2 + t_1t_2$. This yields a unimodular matrix built as in \eqref{eq: M matrix}: 
    $$M = \begin{pmatrix}
        1 & 0 & -1 & -1 & 0\\ 
            0 &1 & 1 & 0 & -1 \\
            0 & 0 & 1 & 1 & 0\\
            0 & 0 & 0 & 0 & 1\\
            0 & 0 & 0 & 1 & 1 
    \end{pmatrix}. $$
    Notice that these polynomials and their Newton polytopes appeared in Example \ref{ex: pentagon intro}.
\end{example}

\begin{example} \label{ex:torsion}
    The normal fan $\Sigma$ of the diamond with vertices $(1,0), (0,1), (-1,0), (0,-1)$ has ray matrix $F = \left ( \begin{smallmatrix}
        1 & -1 & -1 & 1\\ 1 &1 & -1 & -1 
    \end{smallmatrix} \right)$. The divisor class group of the corresponding toric surface $X_\Sigma$ is isomorphic to $\mathbb{Z}^2 \oplus \mathbb{Z} /2 \mathbb{Z}$. Hence, $F$ cannot be extended to a unimodular matrix.
\end{example}

For convenience, we summarize the assumptions on $f_1, \ldots, f_k$ made so far.
\begin{assumption}\label{assump: on polys}
    The Laurent polynomials $f_1, \ldots, f_k \in \mathbb{C}[t_1^{\pm 1}, \ldots, t_d^{\pm 1}]$ satisfy two conditions:
    \begin{enumerate}
        \item The sum $P = P_1 + \cdots + P_k$, with $P_i = {\rm Newt}(f_i)$, is $d$-dimensional and has $d+k$ facets.
        \item The matrix $M$ built as in \eqref{eq: M matrix} is unimodular, i.e. $\det M = \pm 1$.
    \end{enumerate}
\end{assumption}

For statements about nonnegative points, we shall also assume the following. 
\begin{assumption} \label{assump:nonneg}
    The coefficients of $f_1, \ldots, f_k \in \mathbb{C}[t_1^{\pm 1}, \ldots, t_d^{\pm 1}]$ are real and nonnegative.
\end{assumption}

We already saw that Assumption \ref{assump: on polys} limits us to normal projective toric varieties whose class group has no torsion. We now show a stronger consequence. 

\begin{lemma}\label{lem: smooth}
    If $f_1,\ldots, f_k$ satisfy Assumption~\ref{assump: on polys}, then the toric variety $X_\Sigma = X_{\Sigma_P}$ is smooth.
\end{lemma}
\begin{proof}
    The condition that $M$ is unimodular implies that the divisor classes $[D_{P_1}], \ldots, [D_{P_k}] \in {\rm Pic}(X_\Sigma)$ associated to the $f_i$ generate ${\rm Cl}(X_\Sigma) = \mathbb{Z}^n / {\rm im}_{\mathbb{Z}} F^t$. Thus, the Picard group equals the class group, which implies that $X_\Sigma$ is smooth by \cite[Proposition 4.2.6]{CoxLittleSchenckToric}.
\end{proof}

Example \ref{ex:singular} will illustrate the case where $X_\Sigma$ is singular but the class group has no torsion. Next, we build a rational variety $U$ using $f_1, \ldots, f_k$ and $M$. %Let $f_*=f_1\cdots f_k$ denote the product of our Laurent polynomials. 
We start from the map 
\[
    \Gamma_{f^{-1}}: (\mathbb{C}^*)^d \setminus V(f_1 \cdots f_k) \rightarrow (\mathbb{C}^*)^n \qquad t \mapsto (t_1, \ldots, t_d,f_{1}(t)^{-1},\ldots, f_k(t)^{-1}).
\]
The image of $\Gamma_{f^{-1}}$ is a very affine variety denoted by $V^\circ \subset (\mathbb{C}^*)^n$. The map $p: V^\circ \to (\mathbb{C}^*)^d \setminus V(f_1 \cdots f_k)$ denotes its inverse: the projection onto the first $d$ coordinates. For any integer matrix $B \in \mathbb{Z}^{p \times q}$, we write $\phi_B : (\mathbb{C}^*)^p \rightarrow (\mathbb{C}^*)^q$ for the monomial map given by $(t_1, \ldots, t_p) \mapsto(t^{b_1}, \ldots, t^{b_q})$, where the exponent vectors $b_1, \ldots, b_q \in \mathbb{Z}^p$ are the columns of $B$. For properties of $\phi_B$, see \cite[Chapter 1]{Telen2025AppliedToricGeometry}. The monomial map $\phi_{M^t} : (\mathbb{C}^*)^n \rightarrow (\mathbb{C}^*)^n$ is an isomorphism of tori when $M$ is unimodular. Its inverse is given by $\phi_{M^{-t}} : (\mathbb{C}^*)^n \rightarrow (\mathbb{C}^*)^n$. 
We write $U^\circ$ for the very affine variety $ \phi_{M^{-t}}(V^\circ) \subset (\mathbb{C}^*)^n$ and $\varphi$ for the map $\phi_{M^{-t}} \circ \Gamma_{f^{-1}}$. The commutative diagram in Figure~\ref{fig: def of U variety} keeps track of the construction. 
\begin{figure}[h!]
    \centering
    \begin{tikzcd}[column sep=huge, row sep=large]
(\C^*)^d \setminus V(f_1 \cdots f_k)
  \arrow[r,shift left =1, "\Gamma_{f^{-1}}"]
  \arrow[dr,bend right=15,"\varphi"]
&
V^\circ
  \arrow[r,hook] \arrow[l, shift left=1, "p"]
&
(\C^*)^n
\\
&
U^\circ
  \arrow[u]
  \arrow[r,hook]
&
(\C^*)^n
  \arrow[u,"\phi_{M^t}"]
\end{tikzcd}
    \caption{Commutative diagram illustrating the definition of $U$.}
    \label{fig: def of U variety}
\end{figure}
We define the affine variety $U$ as the closure of $U^\circ$ in $\C^n$. %A priori, it is not clear why  $U$ should not intersect the base locus $Z(\Sigma)$. We will prove this at the end of the section in Corollary~\ref{cor: base locus}. 
Here are two examples of this construction. 
\begin{example}
Applying the recipe above to the polynomials $f_1, f_2, f_3$ and the matrix $M$ from Example \ref{ex: pentagon 3} gives the map $\varphi$ and the variety $U$ in Example~\ref{ex: pentagon intro}. 
\end{example}
\begin{example} \label{ex:P1P1chart}
    The polynomials $f_1=1+t_1$ and $f_2=1+t_2$ give rise to the fan $\Sigma$ from Example~\ref{ex:P1P1cox}.  These satisfy Assumptions \ref{assump: on polys} and \ref{assump:nonneg}. We have $X_\Sigma = \P^1 \times \P^1$ and \[ M \, = \, \begin{pmatrix}
        1 & 0 & -1 & 0\\
       0 & 1 & 0 & -1\\
        0 & 0 & 1 & 0\\
        0 & 0 & 0 & 1
    \end{pmatrix}. \] 
    The map $\varphi: (\C^*)^2 \dashrightarrow (\C^*)^4$ is $\varphi(t_1,t_2) = (t_1({1+t_1})^{-1}, {t_2}({1+t_2})^{-1},({1+t_1})^{-1}, ({1+t_2})^{-1}).$ The variety $U \simeq \mathbb{C}^2$ is the affine plane in $\mathbb{C}^4$ defined by $u_1+u_3=u_2+u_4=1$.
\end{example}
\noindent The following result allows us to identify the coordinates of $U^\circ$ as the Cox coordinates of $X_\Sigma$. 
\begin{proposition}\label{prop:inviscox}
    Let $f_1, \ldots, f_k \in \mathbb{C}[t_1^{\pm 1}, \ldots, t_d^{\pm 1}]$ be Laurent polynomials satisfying Assumption \ref{assump: on polys}. The inverse of the map $\varphi: (\mathbb{C}^*)^d \setminus V(f_1 \cdots f_k) \rightarrow U^\circ$ is the restriction of the monomial map $\phi_{F^t}: (\mathbb{C}^*)^n \rightarrow (\mathbb{C}^*)^d$ to ${\rm im} \, \varphi = U^\circ \subseteq (\mathbb{C}^*)^n$. 
\end{proposition}
\begin{proof}
    Recall that $\varphi = \phi_{M^{-t}} \circ \Gamma_{f^{-1}}$. The graph map is clearly invertible with inverse given by the projection $p$ onto the first $d$-coordinates. By construction $\phi_{M^{-t}}$ has inverse $\phi_{M^{t}}$. The proposition follows from the fact that the first $d$ coordinates of $\phi_{M^t}$ are given by $\phi_{F^t}$. 
\end{proof}
Recall that the monomial map $\phi_{F^t}$ is precisely the restriction of the GIT quotient map $\pi: \C^n\setminus Z(\Sigma) \to X_\Sigma$ to the torus $(\C^*)^n$. Proposition \ref{prop:inviscox} says that $\varphi$ is a section of $\pi$: 
\[ \pi \circ \varphi \, = \,  {\rm id}_{(\mathbb{C}^*)^d \setminus V(f_1 \cdots f_k)}. \] Thus, there is exactly one point of $U^\circ$ on each $G$-orbit $\pi^{-1}(p)$, for $p \in (\mathbb{C}^*)^d \setminus V(f_1 \cdots f_k) \subset X_\Sigma$.

%\textcolor{blue}{We need a lemma which says that $U$ is disjoint from $Z(\Sigma)$, so that $\pi_{|U}$ is well-defined. This follows from the defining equations we have.}

We now continue working towards a proof of Theorem~\ref{thm:mainintro} by introducing the variety $\mathcal{U}$. For this, we will remove the vanishing locus of $f_1 \cdots f_k$ from $X_\Sigma$. More precisely, the Newton polytope $P = P_1 + \cdots + P_k$ of $f_1 \cdots f_k$ corresponds to a Cartier divisor 
\[ D_P \, = \,  \sum_{j = 1}^n ( \sum_{i = 1}^k a_{i,j}) D_{j} \, \in \, {\rm Div}_T(X_\Sigma). \] 
Hence, $f_1 \cdots f_k$ can be viewed as a section of the corresponding line bundle, and its divisor of zeros, denoted by $H$, is the closure of $V(f_1 \cdots f_k) \subset (\mathbb{C}^*)^d$ in $X_\Sigma$. By Lemma \ref{lem: smooth}, the polytope $P$ is smooth. In particular, the divisor $D_P$ is very ample. The support of $H$ is a hyperplane section of $X_\Sigma$ in its projective embedding corresponding to $D_P$. This establishes that $\mathcal{U}:=X_\Sigma\setminus H$ is an affine open~subset.

Let $U_{>0} = U \cap \mathbb{R}_{>0}^n$ be the points of $U^\circ$ with positive coordinates. If the Laurent polynomials $f_1,\ldots,f_k$ have nonnegative coefficients, then $U_{>0} \supseteq \varphi(\R_{>0}^d)$. The opposite inclusion follows from  Proposition~\ref{prop:inviscox}, so that $U_{>0} = \varphi(\R_{>0}^d)$. Let $U_{\geq 0}$ be the semialgebraic set $U \cap \mathbb{R}^n_{\geq 0}$.

\begin{theorem}\label{thm: main with polys}
    Let $f_1,\ldots, f_k$ satisfy Assumptions~\ref{assump: on polys} and \ref{assump:nonneg}. Consider $\varphi: (\mathbb{C}^*)^d \setminus V(f_1 \cdots f_k) \rightarrow \mathbb{C}^n$, $U \subset \mathbb{C}^n$, and $\mathcal{U} \subset X_\Sigma$ as built above. The following statements hold: 
    \begin{enumerate}
    \item $U \subset \mathbb{C}^n \setminus Z(\Sigma)$ (see Corollary \ref{cor: base locus}) and $\pi_{|U}: U  \overset{\sim}{\longrightarrow} \mathcal{U}$ is an isomorphism,
    \item $\varphi$ extends to the inverse $\mathcal{U} \overset{\sim}{\longrightarrow} U$ of $\pi_{|U}$, and $\Sigma_{{\rm Newt}(\varphi)} = \Sigma$,
    \item $\pi(U_{\geq 0}) = (X_\Sigma)_{\geq 0}$.
\end{enumerate}
Hence, $U$ is a positive chart of $X_\Sigma$, in the sense of Definition \ref{def:maindef}.
\end{theorem}
\begin{proof}
    The theorem essentially follows from the discussion in \cite[Section 10]{ArkaniHamed2021}. For completeness, we rephrase the argument in the language of toric geometry, using our notation. We start by showing that $\varphi: \mathbb{C}^d \setminus V(f_1 \cdots f_k) \rightarrow U$ extends to a morphism $\varphi: \mathcal{U} \rightarrow U$. The coordinate ring of $U$ is generated by the coordinate functions $\varphi_1, \ldots, \varphi_n$ of $\varphi$. By construction, $\varphi_i$ is a Laurent monomial in the coordinates of $\Gamma_{f^{-1}}$, with exponent given by the $i$-th column of $M^{-t}$. Hence, the coordinate ring  $\mathbb{C}[U]$ is isomorphic to the $\mathbb{C}$-algebra
    \[ \mathbb{C}\big[ \varphi^w \, : \, w \in \mathbb{N}^n \big] \, = \, \mathbb{C}\big[\Gamma_{f^{-1}}^{M^{-t}w} \, : \, w \in \mathbb{N}^n \big] \, = \, \mathbb{C} \Bigg [ \frac{t_1^{\nu_{1}} \cdots t_d^{\nu_{d}}}{f_1^{s_{1}} \cdots f_k^{s_{k}}} \, : \, M^t v \, = \, F^t \nu + \sum_{i = 1}^k s_{i} \, a_i \geq 0    \Bigg ]. \]
    Here $a_1^t, \ldots, a_k^t$ are the last $k$ rows of $M$ and $v$ is the vector in $\C^n$ whose first $d$ entries are $\nu=(\nu_1,\ldots,\nu_d)$ and whose last $k=n-d$ entries are $s=(s_1,\ldots,s_k)$. On the other hand, via the projective embedding of $X_\Sigma$ corresponding to the smooth polytope $P$, the coordinate ring of $\mathcal{U}$ is represented~as  
    \[ \mathbb{C}[\mathcal{U}] \, \simeq \, \mathbb{C} \Big [ \frac{t^\ell}{f_1 \cdots f_k} \, : \, \ell \in P \cap \mathbb{Z}^d \Big ]. \]
    We have now represented $\mathbb{C}[\mathcal{U}]$ and $\mathbb{C}[U]$ as subalgebras of the field of rational functions $\mathbb{C}(t_1, \ldots, t_d)$. We will show that these subalgebras are equal. In particular, each $\varphi_i$ extends to a regular function on $\mathcal{U}$. The inclusion $\mathbb{C}[\mathcal{U}] \subseteq \mathbb{C}[U]$ is clear by identifying the generator $t^\ell/(f_1 \cdots f_k)$ of $\mathbb{C}[\mathcal{U}]$ with the generator corresponding to $\nu = \ell$ and $s = (1, \ldots, 1)$ of $\mathbb{C}[U]$. To show the opposite inclusion, note that for each choice of $(\nu, s)$ satisfying $F^t \nu + \sum_{i = 1}^k s_{i} a_i \geq 0$ there exists a positive integer $r$ and coefficients $c_{\ell'} \in \mathbb{C}$ such that 
    \[ \frac{t_1^{\nu_{1}} \cdots t_d^{\nu_{d}}}{f_1^{s_{1}} \cdots f_k^{s_{k}}} \, = \, \frac{t_1^{\nu_{1}} \cdots t_d^{\nu_{d}}}{f_1^{s_{1}} \cdots f_k^{s_{k}}} \frac{f_1^{r - s_{1}} \cdots f_k^{r - s_{k}}}{f_1^{r - s_{1}} \cdots f_k^{r - s_{k}}} \, = \, \sum_{\ell' \, \in \,  (r \cdot P) \cap \mathbb{Z}^d} c_{\ell'} \,  \frac{t^{\ell'}}{(f_1 \cdots f_k)^{r}}.\]
    By Lemma \ref{lem: smooth}, $P$ is smooth, and hence very ample. Therefore, we can pick $r$ large enough so that every lattice point of $r \cdot P$ can be written as a sum of $r$ lattice points in $P$ \cite[Exercise 4.9]{michalek2018selected}. That is, for each $\ell' \in (r \cdot P) \cap \mathbb{Z}^d$ we can write 
    \[ \frac{t^{\ell'}}{(f_1 \cdots f_k)^{r}} \, = \, \prod_{j = 1}^r \frac{t^{\ell_j}}{f_1 \cdots f_k}\]
    where each $\ell_j$ belongs to $P \cap \mathbb{Z}^d$. This proves the equality $\mathbb{C}[U]  = \mathbb{C}[\mathcal{U}]$ as subalgebras of $\mathbb{C}(t_1, \ldots, t_d)$. In particular, $\varphi: (\mathbb{C}^*)^d \setminus V(f_1 \cdots f_k) \rightarrow U^\circ$ extends to an isomorphism $\mathcal{U} \rightarrow U$. By Proposition \ref{prop:inviscox}, the inverse of that isomorphism agrees with the quotient map $\pi_{|U \setminus Z(\Sigma)}: U \setminus Z(\Sigma) \rightarrow X_\Sigma$ on the dense open subset $U^\circ$ of $U \setminus Z(\Sigma)$. We will see in Corollary \ref{cor: base locus} that $U \cap Z(\Sigma) = \emptyset$, which concludes the proof of point 1. 

    \smallskip

    The claim $\Sigma_{{\rm Newt}(\varphi)} = \Sigma$ in point 2 of the theorem follows directly from the construction of $\varphi$. Indeed, note that each $\varphi_i$ is a Laurent monomial evaluated in $t_1,\ldots,t_d,f_1(t),\ldots,f_k(t)$. Moreover, as $M$ is unimodular, each $f_j$ appears in some $\varphi_i$ as a factor. 

    \smallskip

    Since $f_1,\ldots,f_k$ have nonnegative coefficients, the closure of $V(f_1 \cdots f_k) \subset (\mathbb{C}^*)^d$ does not intersect the nonnegative part of $X_\Sigma$. In other words, $\mathcal{U} \supset (X_\Sigma)_{\geq 0}$. The inclusion $\pi(U_{\geq 0}) \subseteq (X_\Sigma)_{\geq 0}$ is clear. For the opposite inclusion, we use the fact that $(X_\Sigma)_{\geq 0}$ is the Euclidean closure of $(\mathbb{R}^d)_{>0}$ in the real points of $X_\Sigma$. For any $p \in (X_\Sigma)_{\geq 0}$, let $ \gamma: [0,1] \rightarrow (X_\Sigma)_{\geq 0}$ be a smooth path with $\gamma(\varepsilon) \in \mathbb{R}^d_{>0}$ for $\varepsilon \in (0,1]$ and $\gamma(0) = p$. Clearly, for $\varepsilon \in (0,1]$ we have $((\varphi_1 \circ \gamma)(\varepsilon), \ldots, (\varphi_n \circ \gamma)(\varepsilon)) \in U_{\geq 0}$. Then, by continuity of $\varphi$ on $\mathcal{U}$, we have $\varphi(p) = \lim_{\varepsilon \rightarrow 0} ((\varphi_1 \circ \gamma)(\varepsilon), \ldots, (\varphi_n \circ \gamma)(\varepsilon)) \in U_{\geq 0}$. This shows that $\varphi((X_\Sigma)_{\geq 0}) \subseteq U_{\geq 0}$, and applying $\pi$ on both sides we find that $(X_\Sigma)_{\geq 0} \subseteq \pi(U_{\geq 0})$. This concludes the proof of point 3.
\end{proof}

One particular consequence of Theorem~\ref{thm: main with polys} is that the points of $U$ are in one-to-one correspondence with the $G$-orbits $\pi^{-1}(p)$ for $p \in \mathcal{U}$. 

%The toric moment map $\mu: (X_\Sigma)_{\geq 0} \rightarrow P$ identifies $(X_\Sigma)_{\geq 0}$ with the polytope $P$ as a real manifold with boundary \cite[\S 4.2]{Fulton1993}. Point 3 of Theorem \ref{thm: main with polys} has the following corollary.

We have defined the varieties $U^\circ$ and $U$ using the rational parametrization $\varphi$. We now derive their defining equations. For $i = 1, \ldots, k$, let $f_i^h = \eta_{a_i}(f_i)$ be the homogenization of $f_i$ to the Cox ring $S = \mathbb{C}[u_1, \ldots, u_n]$ of $X_\Sigma$, as in \eqref{eq:homogenization}, where $a_i$ is the $(d+i)$-th row of $M$ representing the torus-invariant divisor $D_{P_i}$. 

\begin{proposition} \label{prop:eqsfromhom}
    Let $f_1,\ldots,f_k$ satisfy Assumption~\ref{assump: on polys} and let $U^\circ \subset (\mathbb{C}^*)^n$ be as above. Then, $U^\circ$ is a complete intersection given by the $k$ equations $f_{1}^h(u) = \cdots = f_{k}^h(u) = 1$.
\end{proposition}
\begin{proof}
Let us use coordinates $w_1, \ldots, w_n$ for the torus $(\mathbb{C}^*)^n$ containing the $d$-dimensional very affine variety $V^\circ$. Note that, since $V^\circ$ is defined as the image of the graph map $\Gamma_{f^{-1}}$, it is a complete intersection whose ideal is generated by the $k$ Laurent polynomials
\[ 1 - w_{d+1} \, f_{1}(w_1, \ldots, w_d) \, , \, \ldots \, , \, 1 - w_{n} \, f_{k}(w_1, \ldots, w_d) . \]
Pulling back these equations along the isomorphism $\phi_{M^t}$ gives indeed $1- f^h_{j}(u) = 0$.
\end{proof}

Using the equations of $U^\circ$ we can now prove that $\pi|_U$ is actually a well-defined map.

\begin{corollary}\label{cor: base locus}
    Let $f_1, \cdots, f_k$ satisfy Assumption \ref{assump: on polys}. The affine variety $U \subset \C^n$ constructed above does not intersect the base locus $Z(\Sigma)$.
\end{corollary}
\begin{proof}
    By Proposition~\ref{prop:eqsfromhom} we have that $U \subseteq V_{\mathbb{C}^n}(1- f^h_{i}(u): i=1,\ldots,k)$. It suffices to check that the righthand side does not intersect $Z(\Sigma)$. If $u^* \in Z(\Sigma)$ satisfies $f^h_i(u^*) = 1$ for $i = 1, \ldots,k$, then $\prod_{i = 1}^k f^h_i(u^*) = \eta_{a_1 + \cdots + a_k}(f_1 \cdots f_k) (u^*) = 1$. However, since the divisor $D_{P_1} + \cdots + D_{P_k}$ is ample, each monomial of degree $\alpha = [D_{P_1} + \cdots + D_{P_k}] \in {\rm Cl}(X_\Sigma)$ in the Cox ring $S$ is divisible by a generator of the irrelevant ideal $B(\Sigma)$. Indeed, such monomials correspond to lattice points $\ell$ in $P \cap \mathbb{Z}^d$, and the variable $u_i$ appears in the monomial $m:=u^{F^t \ell + a_1 + \cdots + a_k}$ of $\ell$ if and only if $\ell$ does not lie on the $i$-th facet of $P$. Consequently, $m$ is divisible by each generator $u^{\hat{\sigma}}$ of $B(\Sigma)$ for which the smallest face of $P$ containing $\ell$ also contains the vertex corresponding to the $d$-dimensional cone $\sigma \in \Sigma(d)$. Such a vertex exists for each $\ell$. But this means that $\prod_{i = 1}^k f^h_i \in B(\Sigma)$ and thus $\prod_{i = 1}^k f^h_i(u^*) = 0$, a contradiction. 
\end{proof}

\begin{corollary} \label{cor:saturation}
    Let $f_1,\ldots,f_k$ satisfy Assumption~\ref{assump: on polys} and let $U \subseteq \mathbb{C}^n$ be as above. The prime vanishing ideal $I(U) \subset \mathbb{C}[u_1, \ldots, u_n]$ is given by the saturation 
    \[ \langle f_1^h(u) - 1, \ldots, f_k^h(u) - 1 \rangle : \langle u_1 \cdots u_n \rangle^\infty.\]
\end{corollary}

The papers \cite{ClusterConf, had} feature families of examples for which it is shown that the saturation in Corollary \ref{cor:saturation} is not necessary, i.e., $I(U) =  \langle f_1^h(u) - 1, \ldots, f_k^h(u) - 1 \rangle$. The same is true in all examples we tried, see Section~\ref{sec: examples}. This motivates the following conjecture.
\begin{conjecture}\label{conj: saturation}
    Let $f_1,\ldots,f_k$ satisfy Assumption~\ref{assump: on polys}. The prime vanishing ideal of the variety $U \subseteq \mathbb{C}^n$ constructed as above is generated by the $k$ polynomials $f_{i}^h(u)-1,i=1, \ldots, k$.
\end{conjecture}

\begin{example}
    The defining equations of $U$ in Example \ref{ex: pentagon intro} are obtained by homogenizing $f_1, f_2$ and $f_3$ from Example \ref{ex: pentagon 3}. See Example \ref{ex:homogenizepentagon} for these homogenizations.
\end{example}

We have seen how from $k$ Laurent polynomials $f_1,\ldots,f_k$ satisfying Assumptions \ref{assump: on polys} and \ref{assump:nonneg} we can build varieties $U$ and $\mathcal{U}$ satisfying the three conditions in Theorem \ref{thm:mainintro}. We now show that we can find such Laurent polynomials for any smooth projective fan $\Sigma$.

\begin{proposition} \label{prop:fantopols}
    Let $\Sigma$ be any smooth projective fan in $\mathbb{R}^d$ whose rays span $\mathbb{R}^d$. There exist Laurent polynomials $f_1, \ldots, f_k$ satisfying Assumptions \ref{assump: on polys} and \ref{assump:nonneg}, and such that $\Sigma_P = \Sigma$.
\end{proposition}
\begin{proof}
    The nef cone ${\rm Nef}(\Sigma)$ is a full dimensional cone in $\mathbb{R}^{n-d} = \mathbb{R}^k$. Its lattice points are the basepoint free divisor classes on $X_\Sigma$. By toric resolution of singularities \cite[Theorem 11.1.9]{CoxLittleSchenckToric}, we can find a smooth $k$-dimensional cone $C$ contained in ${\rm Nef}(\Sigma)$. Choose torus-invariant divisors $D_{P_1}, \ldots, D_{P_k} \in {\rm Div}_T(X_\Sigma)$ whose classes in ${\rm Cl}(X_\Sigma) = {\rm Pic}(X_\Sigma)$ are the primitive ray generators of $C$. The global sections of the line bundle ${\cal O}_{X_\Sigma}({D_{P_i}})$ on $X_\Sigma$ corresponding to $D_{P_i}$ are Laurent polynomials with a prescribed Newton polytope $P_i$. We choose $f_i$ among these sections to have nonnegative coefficients standing with all candidate Laurent monomials, and positive coefficients standing with the monomials corresponding to the vertices of $P_i$, so that ${\rm Newt}(f_i) = P_i$. The matrix $M$ from \eqref{eq: M matrix} is unimodular by~construction, since $[D_{P_1}], \ldots, [D_{P_k}]$ generate a smooth cone $C$ in the Picard lattice. The class $[\sum_{i = 1}^k D_{P_i}]$ lies in the interior of $C \subset {\rm Nef}(\Sigma)$, so the normal fan of $P = P_1 + \cdots + P_k$ is $\Sigma$. 
\end{proof}

\begin{proof}[Proof of Theorem \ref{thm:mainintro}]
    The statement follows from Proposition \ref{prop:fantopols} and Theorem \ref{thm: main with polys}.
\end{proof}

\begin{algorithm} \label{algo} The proof of Proposition \ref{prop:fantopols} describes an algorithm for computing a positive chart of $X_\Sigma$. The input is $\Sigma$ or a smooth polytope whose normal fan is $\Sigma$, and the output consists of $f_1,\ldots, f_k$ and $U$ from Theorem \ref{thm: main with polys}. First, compute a smooth $k$-dimensional subcone $C \subseteq {\rm Nef}(\Sigma)$. This can be done by essentially following the algorithm for toric resolution of singularities \cite[Theorem 11.1.9]{CoxLittleSchenckToric}. Second, for the divisors $D_{P_i}$ corresponding to the $k$ primitive ray generators of $C$, choose a positive global section $f_i$ of ${\cal O}_{X_\Sigma}(D_{P_i})$. These $f_i$ satisfy Assumptions \ref{assump: on polys} and \ref{assump:nonneg} by construction. Third, compute the ideal of $U$ via Corollary~\ref{cor:saturation}. 
\end{algorithm}
Notice that one important step in Algorithm \ref{algo} is to compute a smooth subcone $C$ of the nef cone. In combinatorics and polytope theory, studying the nef cone (or type/deformation cone) of a given polytope is an interesting question on its own, see \cite{postnikov2008faces,castillo2022deformation,typeConeAssoc}. We have implemented this algorithm in Julia. Our code is available at \cite{zenodo}. 

\begin{example}
    We have seen in Example \ref{ex: pentagon 2} that the nef cone in our running example is itself a smooth cone, and we also computed its generators. These generators are the divisor classes corresponding to the last three rows in the matrix $M$ from Example \ref{ex: pentagon 3}.
\end{example}

\begin{example} \label{ex:singular}
    The aim of this example is to illustrate a version of our construction for singular toric varieties. In particular, we exemplify how Theorem \ref{thm: main with polys} fails when the Laurent polynomials $f_i$ are sections of non-Cartier divisors which lead to a unimodular extension $M$ of $F$. Let $X_\Sigma$ and $F$ be as in Example \ref{ex:P121}. Adding the row $\begin{pmatrix}
        0 & 0 & 1
    \end{pmatrix}$ to $F$ gives a unimodular matrix $M$. The corresponding torus-invariant divisor $D_3 = 0 D_1 + 0 D_2 + 1 D_3$ is effective, but not Cartier. The global sections of ${\cal O}_{X_\Sigma}(D_3)$ are identified with polynomials of the form $c_0 + c_1 \, t_1$. The monomials $1$ and $t_1$ are the lattice points in the polytope \eqref{eq:P001}. We choose the section with positive coefficients $f_1 = 1 + t_1$ and consider the map $\Gamma_{f^{-1}}: (t_1,t_2) \mapsto (t_1,t_2,(1+t_1)^{-1})$ to complete the diagram in Figure \ref{fig: def of U variety}. The variety $U^\circ$ is the image of the map 
    \[ \varphi(t_1,t_2) \, = \, (\phi_{M^{-t}} \circ \Gamma_{f^{-1}})(t_1,t_2) \, = \, \Big ( \frac{t_1}{1+t_1}, \frac{t_2}{(1+t_1)^2}, \frac{1}{1+t_1} \Big).\]
    Its defining equation is $f_1^h-1 = u_3 + u_1 - 1 = 0$, which illustrates that Proposition \ref{prop:eqsfromhom} holds in this setting. The nonnegative part of $U$ is not compact. In particular, it is not homeomorphic to the triangle $P$ from Example \ref{ex:P121}. The polytope ${\rm Newt}(\varphi)$ of $\varphi$ is a line segment.
\end{example}

\section{Examples}\label{sec: examples}
This section presents examples of the construction outlined in Section \ref{sec:constructing}. We reproduce $u$-equations for two binary geometries from the positive geometry literature: Example \ref{ex:hexagon} appears in \cite{had}, and Example \ref{ex:pezzotope} comes from \cite{PosPezzo}. We reiterate that our running Example \ref{ex: pentagon intro} also appears in \cite{ArkaniHamed2021,binaryGeometries,BrownMZVModuli}, and more examples of binary geometries which fit our framework are found in \cite{HeLiRamanZhang,had,PosPezzo,pellytope}. We have implemented several useful functions for computing with positive charts in a Julia Package named \texttt{PositiveChartsToricVarieties.jl}, which relies heavily on \texttt{Oscar.jl} \cite{OSCAR}. We shall illustrate its functionalities. The code is available at \cite{zenodo}.

\begin{example}\label{ex: squares}
    We start with positive charts of the surface $\P^1 \times \P^1$. Recall from Example \ref{ex:P1P1cox} that $\P^1 \times \P^1\cong X_\Sigma$ where $\Sigma$ is the normal fan  of the unit square $P = [0,1]^2 \in \mathbb{R}^2$. Its ray matrix is $F = \left ( \begin{smallmatrix}
        1 & 0 & -1 & 0 \\ 0 & 1 & 0 & -1
    \end{smallmatrix} \right )$.
    Choosing the basis $[D_3],[D_4]$ for ${\rm Pic}(X_\Sigma)$, the nef cone $\Nef(\Sigma)$ is the positive orthant in $\R^2$. For any choice of nonnegative integer matrix $E = \left(\begin{smallmatrix}
        a & b \\ c & d
    \end{smallmatrix}\right)$ with determinant $1$ we have that the cone $C_E = \Cone((a,b),(c,d))$ is a smooth subcone of ${\rm Nef}(\Sigma)$. The rays of $C_E$ correspond to weak Minkowski summands of the unit square: an $a \times b$ rectangle and a $c \times d$ rectangle. We extend $F$ to a unimodular $4 \times 4$ matrix as follows:
    \[
        M = \begin{pmatrix}
            1 & 0 & -1 & 0\\
            0 & 1 & 0 & -1\\
            0 & 0 & a & b \\
            0 & 0 & c & d
        \end{pmatrix}, \quad \text{and} \quad M^{-1} = \begin{pmatrix}
            1 & 0 & d & -b\\
            0 & 1 & -c & a\\
            0 & 0 & d & -b \\
            0 & 0 & -c & a
        \end{pmatrix}.
    \]
    We choose two polynomials $f_1= (1+t_1)^a(1+t_2)^b \in H^0(X_\Sigma, {\cal O}_{X_\Sigma}(aD_3 + bD_4))$ and $f_2 = (1+t_1)^c(1+t_2)^d \in H^0(X_\Sigma,{\cal O}_{X_\Sigma}(cD_3 + dD_4))$. The map $\varphi: (\C^*)^2 \setminus V(f_1f_2) \to U^\circ$ is
    \[
        \varphi(t_1,t_2) \, = \, \left( \frac{t_1}{1+t_1}, 
        \frac{t_2}{1+t_2}, 
        \frac{1}{1+t_1}, 
        \frac{1}{1+t_2} \right).
    \]
    Interestingly, this does not depend on the choice of the nonnegative unimodular matrix $E$. The polynomials $f_1^h$ and $f_2^h$ are bihomogeneous forms in $S = \mathbb{C}[u_1,u_3] \otimes \mathbb{C}[u_2,u_4]$:  
    \[
        f_1^h = (u_3+u_1)^a(u_2+u_4)^b \qquad f_2^h = (u_3+u_1)^c(u_2+u_4)^d . 
    \]
    The ideal $\langle f_1^h -1, f_2^h-1 \rangle$ defines the positive chart $u_1 + u_3 = u_2 + u_4 = 1$ from Example~\ref{ex:P1P1chart}.
\end{example}

\begin{example} \label{ex:hexagon}
Let $P \subset \mathbb{R}^2$ be the hexagon with vertices $(0,0),(0,2),(1,3),(2,0),(3,1), (3,3)$. The ray matrix of its normal fan $\Sigma$ is $\left(\begin{smallmatrix}
    0 & -1 & 1 & 0 & -1 & 1\\
    -1 & 0 & 0 & 1 & 1 & -1
\end{smallmatrix}\right)$. 
The cone ${\rm Nef}(\Sigma)$ has ray generators
\[
    r_1 = [D_5+D_6], \, 
    r_2 = [D_1+D_6], \,  
    r_3 =[D_1+D_2+D_5], \, 
    r_4 = [D_1+D_2], \, 
    r_5 = [D_1+D_2+D_6].
\]
This is computed by applying the function \texttt{nef\_cone\_modulo\_lineality} in our package to the hexagon $P$. The cone $\Cone(r_1,r_2,r_3,r_5)$ is a smooth subcone of $\Nef(\Sigma)$. The choice of polynomials $(f_1,f_2,f_3,f_4) = (1 + t_1, 1 + t_2, 1 + t_1 + t_1t_2, 1 + t_2 + t_1t_2)$ yields the positive chart 
\[
     u_2u_5 + u_3u_6 \, = \,  u_1u_6 + u_4u_5\, = \,  u_1u_2u_5 + u_1u_3u_6 + u_3u_4u_5\, = \,  u_1u_2u_6 +u_2u_4u_5 + u_3u_4u_6\, = \, 1.
\]
This variety has degree 10. It also appears in \cite[Example 13.2]{had} and can be presented as a binary geometry \cite[Definition 2.1]{LamModuli} as follows:
\[  u_1 + u_3u_4u_5 \, = \, u_2 + u_3u_4u_6\, = \, u_3 + u_1u_2u_5 \, = \, u_4 + u_1u_2u_6\, = \, 
    u_5 + u_1u_3u_6^2 \, = \,  u_6 + u_2u_4u_5^2 \, = \, 1.\]
A different smooth subcone of $\Nef(\Sigma)$ yields a different variety. The positive chart obtained from the smooth cone $\Cone(r_1,r_2,r_3,r_4) \subset {\rm Nef}(\Sigma)$ is a complete intersection of degree $7$:
\[
    u_2u_5+u_3u_6\, = \, u_1u_6+u_4u_5\, = \, u_1u_2u_5 + u_1u_3u_6 + u_3u_4u_5 \, = \, u_1u_2+u_3u_4 \, = \, 1.
\]
This variety also admits a binary presentation: the six equations
\[  u_1 + u_3u_4u_5 \, = \, u_2 + u_3^2u_4u_6\, = \,u_3 + u_1u_2u_5 \, = \, u_4 + u_1^2u_2u_6\, = \, 
    u_5 + u_1u_3u_6 \, = \, u_6 + u_2u_4u_5^2 \, = \, 1\]
generate the same ideal. Note that the combinatorics of the binary presentations for these positive charts is the same, as the fan $\Sigma$ is the same for both, but the exponents appearing in the second term of each equation are different. 
In both cases, Conjecture~\ref{conj: saturation} holds. 
\end{example}

Binary presentations like those in Example \ref{ex:hexagon} were a main motivation for this project. At the moment, it is not clear to us how to systematically find a binary presentation for the vanishing ideal of $U^\circ$ from the generators in Proposition \ref{prop:eqsfromhom}, or how to decide whether such a binary presentation exists. However, our framework reproduces many binary geometries coming from specific fans appearing in the physics literature. A binary presentation for Example \ref{ex: pentagon intro} is found in \cite[Equation (0.1)]{LamModuli}. Example \ref{ex:pezzotope} provides another instance.

\begin{example} \label{ex:pezzotope}
Next, we compute a positive chart of a toric fourfold arising in positive del Pezzo geometry \cite{PosPezzo}. The \emph{$E_6$ Pezzotope} is the polytope $P = \sum_{i=1}^{11} \Newt(f_i)$,~where 
\begin{align*}
&f_1 = 1+t_1, &&\\
&f_2 = 1+t_2, &&f_7 = 1+t_4+t_2 t_4+t_1 t_2 t_4,\\
&f_3 = 1+t_2+t_1 t_2, &&f_8 = 1+t_3+t_3 t_4+t_2 t_4+t_2 t_3+t_2 t_4 t_3,\\
&f_4 = 1+t_3, &&f_9 = 1+t_3+t_3 t_4+t_2 t_4+t_2 t_3+t_1 t_2 t_4+t_2 t_3 t_4,\\ 
&f_5 = 1+t_3+t_2 t_3, &&f_{10} = 1+t_3+t_2 t_3+t_2 t_4+t_3t_4+t_1 t_2 t_4+t_2 t_3 t_4+t_1 t_2 t_3 t_4,\\
&f_6 = 1+t_4, &&f_{11} = 1+t_3+t_2 t_3+t_2 t_4+t_3t_4+t_1 t_2 t_3+t_1 t_2 t_4+t_2 t_3 t_4+t_1 t_2 t_3 t_4.
\end{align*} 
The $f$-vector of $P$ is $(45,90,60,15)$. The divisors $D_{P_i}$ associated with the polytopes $P_i = {\rm Newt}(f_i)$ are the rays of the smooth rational cone ${\rm Nef}(\Sigma_P)$. The ray matrix of the fan $\Sigma_P$ is 
\setcounter{MaxMatrixCols}{20}
\[ F\, = \, \begin{pmatrix}
    0 & 0 & 0 & 0 & -1 & 0 & 0 & -1 & -1 & 1 & 0 & 0 & -1 & -1 & 0 \\
    0 & 0 & 0 & -1 & 0 & 0 & -1 & 0 & 1 & 0 & 1 & -1 & 0 & 0 & -1 \\
    0 & 0 & -1 & 1 & 0 & 1 & 0 & 0 & 0 & 0 & 0 & 0 & -1 & 1 & 1 \\
    1 & -1 & 0 & 1 & 0 & 0 & 0 & 1 & 0 & 0 & 0 & 1 & 0 & 1 & 0
\end{pmatrix}. \] 
We read from the rows of $F$ that, in order to homogenize $f_1,\ldots,f_{11}$, we replace
\[
    t_1 = \frac{u_{10}}{u_5u_8u_9u_{13}u_{14}}, \qquad t_2 = \frac{u_9u_{11}}{u_4u_7u_{12}u_{15}}, \qquad t_3 = \frac{u_4u_6u_{14}u_{15}}{u_3u_{13}}, \qquad t_4 = \frac{u_1u_4u_8u_{12}u_{14}}{u_2}
\]
and take the numerator of each expression. This yields 11 generators $f_1^h-1, \ldots, f_{11}^h-1$ of $I(U^\circ)$. We verified that Conjecture~\ref{conj: saturation} holds in this case. That is, $U$ is a complete intersection in $\C^{15}$ defined by these 11 equations. Moreover, we also checked that $U$ can be presented as a binary geometry: its ideal is alternatively generated by the $15$ polynomials
\begin{align*}
u_1 + u_2 u_5 u_7 u_{13} u_{15} - 1, && u_6 + u_3 u_7 u_8 u_{12} u_{13} - 1, && u_{11} + u_4 u_5 u_7 u_8 u_{12} u_{13} u_{14} u_{15} - 1, \\
u_2 + u_1 u_4 u_8 u_{12} u_{14} - 1, && u_7 + u_1 u_6 u_9 u_{11} u_{14} - 1 , && u_{12}+u_2u_5u_6u_9u_{11}u_{13}u_{14}u_{15}-1,\\
u_3 + u_4 u_5 u_6 u_{14} u_{15} - 1, && u_8 + u_2 u_6 u_{10} u_{11} u_{15} - 1,&& u_{13} + u_1 u_4 u_6 u_{10} u_{11} u_{12} u_{14} u_{15} - 1 ,\\
u_4 + u_2 u_3 u_9 u_{11} u_{13} - 1, && u_9 + u_4 u_7 u_{10} u_{12} u_{15} - 1, && u_{14} + u_2 u_3 u_7 u_{10} u_{11} u_{12} u_{13} u_{15} - 1 ,\\
u_5 + u_1 u_3 u_{10} u_{11} u_{12} - 1,&& u_{10} + u_5 u_8 u_9 u_{13} u_{14} - 1, && u_{15} + u_1 u_3 u_8 u_9 u_{11} u_{12} u_{13} u_{14} - 1.
\end{align*}
These are the defining equations found in \cite{PosPezzo} for a partial compactification of the moduli space $U^\circ \simeq Y(3,6)$ of marked smooth cubic surfaces. 
\end{example}

\begin{example}
    The $3$-dimensional permutahedron ${\rm Perm}(3)$ is the convex hull of the $24$ points in $\mathbb{R}^4$ whose coordinates are the permutations of $(1,2,3,4)$. The nef cone of this realization is an $11$-dimensional cone with $37$ rays. Using the polytope ${\rm Perm}(3)$ as an input to the function \texttt{unimod\_nef\_polynomials} in our Julia package \texttt{PositiveChartsToricVarieties.jl}, a smooth subcone $C$ of $\Nef({\rm Perm}(3))$ is automatically selected. The function returns the list
    \begin{align}\label{eq: polys perm3}
 &1 + t_1, \qquad  1 + t_2, \qquad 1 + t_3, \qquad t_1 + t_2, \qquad  t_2 + t_3, \qquad  1 + t_2 + t_3,\nonumber\\
 &  t_1 + t_3 + t_1t_3, \qquad t_1 + t_2 + t_1t_2 + t_1t_3 + t_2t_3, \qquad t_1 + t_3 + t_1t_2 + t_1t_3 + t_2t_3,\\
 & t_1t_2 + t_1t_3 + t_2t_3, + t_1t_2t_3,\qquad t_1 + t_2 + t_3 +t_1t_2 + t_1t_3 + t_2t_3 + t_3^2. \nonumber
 \end{align}
    These are $f_1, \ldots, f_{11}$. The divisors of their Newton polytopes correspond to the rays of $C$. Another output is the unimodular matrix $M$ associated with this choice of polynomials: 
 \begin{equation*}\label{eq: matrix perm3}
 	M = \scalebox{0.75}{$\left(\begin{array}{cccccccccccccc}
	  -1 & -1 & -1 & -1 &  1 &  1 & 1 &  0 &  0 &  1 & 0 &  0 &  0 & 0\\
 0 & -1 & -1 &  0 &  0 &  1 & 0 & -1 &  1 &  1 & 1 & -1 &  0 & 0\\
 0 &  0 & -1 & -1 &  1 &  0 & 0 & -1 &  1 &  1 & 0 &  0 & -1 & 1\\
 1 &  1 &  1 &  1 &  0 &  0 & 0 &  0 &  0 &  0 & 0 &  0 &  0 & 0\\
 0 &  1 &  1 &  0 &  0 &  0 & 0 &  1 &  0 &  0 & 0 &  1 &  0 & 0\\
 0 &  0 &  1 &  1 &  0 &  0 & 0 &  1 &  0 &  0 & 0 &  0 &  1 & 0\\
 1 &  1 &  1 &  1 &  0 & -1 & 0 &  1 &  0 & -1 & 0 &  1 &  0 & 0\\
 0 &  1 &  1 &  1 &  0 &  0 & 0 &  1 & -1 & -1 & 0 &  1 &  1 & 0\\
 0 &  1 &  1 &  1 &  0 &  0 & 0 &  1 &  0 &  0 & 0 &  1 &  1 & 0\\
 1 &  1 &  2 &  2 & -1 &  0 & 0 &  1 &  0 & -1 & 0 &  0 &  1 & 0\\
 1 &  2 &  2 &  2 &  0 & -1 & 0 &  2 &  0 & -1 & 0 &  1 &  1 & 0\\
 1 &  2 &  2 &  2 & -1 &  0 & 0 &  2 &  0 & -1 & 0 &  1 &  1 & 0\\
 1 &  2 &  3 &  2 & -1 & -1 & 0 &  2 & -1 & -2 & 0 &  1 &  1 & 0\\
 1 &  2 &  2 &  2 &  0 &  0 & 0 &  2 &  0 & -1 & 0 &  1 &  2 & 0
	\end{array}\right)$}.
\end{equation*}
The first three rows form the $F$-matrix. The variety $U^\circ \subset (\mathbb{C}^*)^{14}$ is three-dimensional and cut out by the homogenizations of the polynomials in \eqref{eq: polys perm3}. These are obtained by replacing
\[
	t_1 \rightarrow \frac{u_5u_6u_7u_{10}}{u_{1}u_2u_3u_4} \qquad t_2 \rightarrow \frac{u_6u_9u_{10}u_{11}}{u_2u_3u_8u_{12}} \quad t_3 \rightarrow \frac{u_5u_9u_{10}u_{14}}{u_3u_4u_8u_{13}}
\]
and taking the numerator of the obtained rational functions. The ideal of $U^\circ$ can be directly obtained from the polytope ${\rm Perm}(3)$ by using the function \texttt{Y\_variety} in our Julia package. Saturating by the product $u_1 \cdots u_{14}$ yields the ideal of its closure $U \subset \mathbb{C}^{14}$, a positive chart on the permutohedral variety $X_{{\rm Perm}(3)}$. The saturation did not terminate within a reasonable amount of time, so we could not verify Conjecture~\ref{conj: saturation} in this example. 
\end{example}

\section{Moment maps} \label{sec:moment}

Moment maps play a central role in positive toric geometry, see \cite[\S 4.2]{Fulton1993}, \cite{Sottile2003} or \cite[Chapter 5]{Telen2025AppliedToricGeometry}. These maps identify the nonnegative part of a projective toric variety with their polytope. More precisely, a moment map is a homeomorphism $\mu: (X_\Sigma)_{\geq 0} \overset{\simeq}{\longrightarrow} P$, where $P$ is such that $\Sigma_P$ equals $\Sigma$. Our goal in this section is to associate a natural moment map to a positive chart $U$, whose nonnegative part $U_{\geq 0}$ can be identified with $(X_\Sigma)_{\geq 0}$ by definition. We provide morphisms $\mu_U: U \rightarrow \mathbb{C}^n$, such that $(\mu_U)_{|U_{\geq 0}}$ is a moment map. 

Let $f_1, \ldots, f_k \in \mathbb{C}[t_1^{\pm 1}, \ldots, t_d^{\pm 1}]$ be Laurent polynomials satisfying Assumptions \ref{assump: on polys} and \ref{assump:nonneg}. Let $U$ be the corresponding positive chart of $X_\Sigma$ as in Theorem \ref{thm: main with polys}. The Newton polytope of $f_i$ is $P_i = {\rm Newt}(f_i)$, and we write $f_i = \sum_{m \in P_i \cap \mathbb{Z}^d} c_{i,m} \, t^m$, where $c_{i,m} \in \mathbb{R}_{\geq 0}$ by Assumption \ref{assump:nonneg}. For any choice of $k$ complex numbers $s_1, \ldots, s_k$, we define the following morphism: 
\begin{equation} \label{eq:muYs} \mu_{U,s} : U \rightarrow \mathbb{C}^n, \quad \quad  u \longmapsto \sum_{i = 1}^k s_i \sum_{m \in P_i \cap \mathbb{Z}^d} c_{i,m} \, u^{\eta_i(m)} \, \eta_i(m) \end{equation}
where $\eta_i(m) = F^t m + a_i \in \mathbb{Z}^{n}_{\geq 0}$ are the exponents of $f_i^h(u)$, the homogenization of $f_i$ to the Cox ring of $X_\Sigma$. It is convenient to introduce the following block row/column notation for the unimodular matrix $M$ from \eqref{eq: M matrix} and its inverse $M^{-1}$: 
\[ M \, = \, \begin{pmatrix}
    F \\ A
\end{pmatrix}, \quad M^{-1} \, = \, \begin{pmatrix}
    B & K
\end{pmatrix}, \quad \text{where} \quad F \in \mathbb{Z}^{d \times n}, \, A \in \mathbb{Z}^{k \times n}, \, B \in \mathbb{Z}^{n \times d}, \, K \in \mathbb{Z}^{n \times k}. \]
Since $M \cdot M^{-1} = M^{-1}\cdot M = {\rm id}_{n\times n}$, we have the following identities: 
\begin{equation} \label{eq:matidentities}FB = {\rm id}_{d \times d}, \quad FK = 0_{d \times k}, \quad  AB = 0_{k \times d}, \quad AK = {\rm id}_{k \times k}, \quad BF + KA = {\rm id}_{n \times n}  . \end{equation} 
In particular, the notation for $K$ is consistent with Section \ref{sec:prelim}, in that $K$ is a Gale dual matrix of $F$.  Our main result says that $\mu_{U,s}$ has the desired property. 

\begin{theorem} \label{thm:momentmap}
    For any $s = (s_1, \ldots, s_k) \in \mathbb{R}^k_{>0}$, the restriction $(\mu_{U,s})_{|U_{\geq 0}}$, with $\mu_{U,s}$ as in \eqref{eq:muYs}, is a homeomorphism whose image is the $d$-dimensional polytope 
    \begin{equation} \label{eq:imagemu} \mu_{U,s}(U_{\geq 0}) \, = \, \{ x \in \mathbb{R}^n \, : \, K^t x = s \text{ and } x \geq 0 \}. \end{equation}
\end{theorem}

\begin{example} \label{ex:pentagonmoment}
    The map $\mu_{U,s}: U \rightarrow \mathbb{C}^5$ in our running example is given by 
    \[ \begin{pmatrix} 0 \\ 0 \\ s_1 \\ s_1 \\ 0 \end{pmatrix} u_3u_4 +
    \begin{pmatrix} s_1 \\ 0 \\ 0 \\ 0 \\ 0 \end{pmatrix} u_1 + 
    \begin{pmatrix} 0 \\ 0 \\ 0 \\ 0 \\ s_2 \end{pmatrix} u_5 + 
    \begin{pmatrix} 0 \\ s_2 \\ s_2 \\ 0 \\ 0 \end{pmatrix} u_2u_3 + 
    \begin{pmatrix} 0 \\ 0 \\ 0 \\ s_3 \\ s_3 \end{pmatrix} u_4u_5 +
    \begin{pmatrix} 0 \\ s_3 \\ s_3 \\ s_3 \\ 0 \end{pmatrix} u_2u_3u_4 +
    \begin{pmatrix} s_3 \\ s_3 \\ 0 \\ 0 \\ 0 \end{pmatrix} u_1u_2.
    \]
    For positive $s_1, s_2, s_3$, the image of $U_{\geq 0}$ under this map is a pentagon in $\mathbb{R}^5$ contained in the plane $K^t x = s$, with $K \in \mathbb{Z}^{5 \times 3}$ as in Example \ref{ex:homogenizepentagon}.
\end{example}

Before proving Theorem \ref{thm:momentmap}, let us establish that the polytope in \eqref{eq:imagemu} indeed has the correct normal fan. Let $P(s) \subset \mathbb{R}^d$ be the polytope obtained as the weighted Minkowski sum $P(s) = s_1 \cdot P_1 + \cdots + s_k \cdot P_k$. Since $s_i > 0, i = 1, \ldots, k$ and $[D_1], \ldots, [D_k]$ span a $k$-dimensional subcone of ${\rm Nef}(\Sigma)$, the normal fan of $P(s)$ is $\Sigma$. We embed $P(s)$ in $\mathbb{R}^n$ via $\iota_s: \mathbb{R}^d \rightarrow \mathbb{R}^n, \,  \nu \mapsto (\nu, s)$. Using \eqref{eq:matidentities}, one checks that the righthand side of \eqref{eq:imagemu} is $M^t \iota_s(P(s))$. 

\begin{proof}[Proof of Theorem \ref{thm:momentmap}]
Our strategy is to relate $\mu_{U,s}$ to more standard moment maps in toric geometry. First, we observe that the restriction to $U^\circ$ of $\mu_{U,s}$ is given by the composition $(\mu_{U,s})_{|U^\circ} \, = \, M^t \circ \mu_{{\rm Cay}} \circ (\phi_{M^t})_{|U^\circ}$, where $\phi_{M^t}: (\mathbb{C}^*)^n \rightarrow (\mathbb{C}^*)^n$ is the Laurent monomial map given by the rows of $M$, $M^t: \mathbb{C}^n \rightarrow \mathbb{C}^n$ is a $\mathbb{C}$-linear map, and $\mu_{{\rm Cay}}$ is given~by 
    \[ \mu_{{\rm Cay}} \, : \, (\mathbb{C}^*)^n \rightarrow \mathbb{C}^n, \quad (w_1, \ldots, w_n) \longmapsto \sum_{i = 1}^k s_i \sum_{m \in P_i \cap \mathbb{Z}^d} c_{i,m} \, w_{1:d}^m \, w_{d+i} \, \begin{pmatrix}
    m \\ e_i 
\end{pmatrix}. \]
Here $e_i \in \mathbb{R}^k$ is the $i$-th standard basis vector, and $w_{1:d} = (w_1, \ldots, w_d)$. For the expert reader, we point out that $\mu_{{\rm Cay}}$ is the affine algebraic moment map associated with the Cayley configuration of $f_1, \ldots, f_k$. Recall that $U^\circ$ is parametrized by $\varphi: (\mathbb{C}^*)^d \setminus V(f_1 \cdots f_k) \rightarrow U^\circ$, and that $\phi_{M^t} \circ \varphi$ is given by $t \mapsto (t_1, \ldots, t_d, f_1^{-1}(t), \ldots, f_k^{-1}(t))$, see Figure \ref{fig: def of U variety}. Hence, on the intersection of $U$ with the torus of $X_\Sigma$, i.e., on $(\mathbb{C}^*)^d \setminus V(f_1 \cdots f_k) \cong U^\circ \subset U$, we have
\[ (M^{-t} \circ \mu_{U,s}) (t) \, = \, \sum_{i = 1}^k s_i \sum_{m \in P_i \cap \mathbb{Z}^d} \frac{c_{i,m} \, t^m}{f_i(t)} \begin{pmatrix}
    m \\ e_i
\end{pmatrix} \, = \, \begin{pmatrix}
    s_1 \, \mu_1(t) + \cdots + s_k \, \mu_k(t) \\ s
\end{pmatrix} \]
where $\mu_i(t) = \sum_{m \in P_i \cap \mathbb{Z}^d} c_{i,m} \, t^m \, f_i(t)^{-1} \, m$. The map $\mathbb{R}^d_{>0} \rightarrow \mathbb{R}^d$ given by $t \mapsto s_1 \, \mu_1(t) + \cdots + s_k \, \mu_k(t)$ appears in \cite[Lemma 2.7]{MatsubaraHeoMizeraTelen2023EulerIntegrals}, where it is shown to be a diffeomorphism onto the interior of $P(s)$.
We must show that its extension $\tilde{\mu}: U \rightarrow \mathbb{C}^d$, given by the first $d$ coordinates of $(M^{-t} \circ \mu_{U,s})$, restricts to a one-to-one map $U_{\geq 0} \rightarrow P(s)$. We shall therefore describe the restriction of $\tilde{\mu}$ to each of the toric boundary strata of $U$. 

\smallskip

Each face $Q$ of the polytope $P(s) = s_1P_1 + \cdots + s_kP_k$ corresponds to a torus orbit $O_Q \subset X_\Sigma$. %The affine open subset $\mathcal{U} \subset X_\Sigma$ is isomorphic to $U$ (Theorem \ref{thm: main with polys}), and it is decomposed as $\mathcal{U} = \bigsqcup_{Q \preceq P} O_Q \cap \mathcal{U}$, where the disjoint union runs over all faces $Q$ of $P$. 
Fix a face $Q$ and let $w \in \mathbb{Z}^d$ be a weight vector which ``exposes'' $Q$, meaning that 
\[ Q \, = \, \{ m \in P \, : \, w \cdot m \, = \, \min_{m' \in P} w \cdot m' \}.\]
Equivalently, $w$ is a lattice point contained in the relative interior of the cone of $\Sigma$ corresponding to $Q$. We write $P_i^w$ for the face of $P_i$ exposed by $w$, and we have $Q = s_1 \, P_1^w + \cdots + s_k \, P_k^w$. Each point $p \in O_Q$ is obtained as a one-parameter limit of the following form: 
\[ p \, = \, \lim_{\varepsilon \to 0} \, \varepsilon^w \cdot t \, = \, \lim_{\varepsilon \to 0} \,  (\varepsilon^{w_1} \, t_1, \ldots, \varepsilon^{w_d} \, t_d). \]
Writing $f_i^w(t) = \sum_{m \in P_i^w} c_{i,m} t^m$, and assuming that $p \in \mathcal{U} \cap O_Q$, we have that  
\[ \mu_i(\varphi(p)) \, = \, \lim_{\varepsilon \rightarrow 0} \mu_i(\varepsilon^w \cdot t) \, = \,  \lim_{\varepsilon \rightarrow 0} \sum_{m \in P_i \cap \mathbb{Z}^d} \frac{c_{i,m} \, \varepsilon^{w \cdot m} \,  t^m}{f_i(\varepsilon^w \cdot t)} \, m \, = \, \sum_{m \in P_i^w \cap \mathbb{Z}^d} \frac{c_{i,m} \, t^m}{f_i^w(t)} \, m. \]
Here $\mathcal{U}$ is as in Theorem \ref{thm: main with polys}, and $p \in \mathcal{U} \cap O_Q$ ensures that the denominator does not vanish. Choosing coordinates $\tilde{t} = (\tilde t_1, \ldots, \tilde t_{\dim Q})$ on $\mathcal{U} \cap O_Q \simeq (\mathbb{C}^*)^{\dim Q} \setminus V(f_1^w \cdots f_k^w)$, we see that the restriction of $\tilde{\mu}$ to $\varphi(\mathcal{U} \cap O_Q)$ is again of the form $\tilde t \mapsto s_1 \mu_1^Q(\tilde t) + \cdots + s_k \mu_k^Q (\tilde t)$, where $\mu_i^Q$ is the moment map associated with the positive Laurent polynomial $f_i^w$. By \cite[Lemma 2.7]{MatsubaraHeoMizeraTelen2023EulerIntegrals}, the map $\tilde \mu_{|\varphi(\mathcal{U} \cap O_Q)}$ restricts to a diffeomorphism $\varphi(\mathcal{U} \cap O_Q) \cap U_{\geq 0} \overset{\simeq}{\longrightarrow} Q$.

\smallskip

The torus orbit stratification $X_\Sigma = \bigsqcup_{Q \preceq P(s)} O_Q$ induces a stratification of the semialgebraic set $U_{\geq 0} = \bigsqcup_{Q \preceq P}\varphi(\mathcal{U} \cap O_Q)  \cap U_{\geq 0}$. Here the disjoint union runs over all faces $Q$ of $P(s)$. The map $\tilde \mu$ identifies each stratum diffeomorphically with the corresponding face of $P(s)$. In that sense, $\tilde{\mu}$ identifies $U_{\geq 0}$ and $P(s)$ as real analytic manifolds with corners. 
%Observe that $\mu_i: \mathbb{R}^d_{>0} \rightarrow \mathbb{R}^d$ is a standard toric moment map associated with the positive Laurent polynomial $f_i$. More precisely, the Laurent polynomial $f_i$ gives an algebraic toric moment map $(X_{\Sigma_{P_i}})_{\geq 0} \rightarrow P_i$ as in \cite[\S 4.2]{Fulton1993}. We interpret $\mu_i$ as the pullback of this map along the surjective toric morphism $X_\Sigma \rightarrow X_{\Sigma_{P_i}}$, induced by the fact that $\Sigma$ is a refinement of $\Sigma_{P_i}$ \cite[p.~130]{CoxLittleSchenckToric}. In particular, $\mu_i$ is well-defined on $(X_\Sigma)_{\geq 0} \simeq U_{\geq 0}$ and $\mu_i((X_\Sigma)_{\geq 0}) = P_i$. 
\end{proof}

\begin{example}
    Consider the matrix $E = \left(\begin{smallmatrix}
        2 & 1 \\ 1 & 1
    \end{smallmatrix}\right)$. The moment map corresponding to the positive chart of $\P^1 \times \P^1$ obtained in Example \ref{ex: squares} is given by the map $\mu_{U,s}: U \to \C^4$:
    \[
(u_1,u_2,u_3,u_4) \, \mapsto \, \begin{pmatrix}
2s_1 u_1^2 u_2 + s_2 u_1 u_2
+ 2s_1 u_1^2 u_4 + s_2 u_1 u_4
+ 2s_1 u_1 u_2 u_3
+ 2s_1 u_1 u_3 u_4 \\
s_1 u_1^2 u_2 + s_2 u_1 u_2
+ 2s_1 u_1 u_2 u_3
+ s_1 u_2 u_3^2 + s_2 u_2 u_3 \\
2s_1 u_1 u_2 u_3
+ 2s_1 u_1 u_3 u_4
+ 2s_1 u_2 u_3^2 + s_2 u_2 u_3
+ 2s_1 u_3^2 u_4 + s_2 u_3 u_4 \\
s_1 u_1^2 u_4 + s_2 u_1 u_4
+ 2s_1 u_1 u_3 u_4
+ s_1 u_3^2 u_4 + s_2 u_3 u_4
\end{pmatrix}. \qedhere
\]

\end{example}

\medskip

We end the section by describing the fibers of the moment map $\mu_{U,s}$ as the solutions to a set of critical point equations, called the \emph{scattering equations} in the physics literature \cite[Section 5.3]{LamModuli}. 
Let $x = (x_1, \ldots, x_n) \in \mathbb{C}^n$ be a tuple of complex exponents. We compute the critical points of the multivalued function $\log(u^x)$ on $U^\circ$, where $u^x = \prod_{\rho \in \Sigma(1)} u_\rho^{x_\rho} = u_1^{x_1} \cdots u_n^{x_n}$. Using Lagrange multipliers $s_1, \ldots, s_k$ standing with the $k$ equations for $U^\circ$ from Proposition \ref{prop:eqsfromhom}, the critical points are determined by the vanishing of the $n + k$ partial derivatives of 
\[ {\cal L} \, = \, \sum_{i = 1}^n x_i \, \log u_i - \sum_{j = 1}^k s_j \, (f_j^h(u) -1).\]
The $n$ equations $u_i \frac{\partial}{\partial u_i} {\cal L} = 0$ can conveniently be written as $\mu_{U,s}(u) = x$, where $\mu_{U,s}$ is our moment map from \eqref{eq:muYs}. The $k$ equations $\frac{\partial}{\partial s_j} {\cal L} = 0$ are written as $u \in U^\circ$. We now eliminate the Lagrange multipliers $s$ by observing that the image of $\mu_{U,s} : U \rightarrow \mathbb{C}^n$ is contained in the $d$-plane $K^t x = s$ by Theorem \ref{thm:momentmap}. We arrive at the following set of \emph{scattering equations}: 
\begin{equation}  \label{eq:scatteqs} x \, = \, \mu_{U, K^t x}(u) \quad \text{and} \quad u \in U^\circ. \end{equation}
For generic $x \in \mathbb{C}^n$, the scheme defined by these equations consists of $(-1)^d \cdot \chi(U^\circ)$ reduced points \cite[Theorem 1]{Huh2013}. Here $\chi(\cdot)$ is the topological Euler characteristic. 

\begin{example}
The scattering equations for $U$ from Example \ref{ex: pentagon intro} are read from the moment map in Example \ref{ex:pentagonmoment}, with $(s_1,s_2,s_3)$ replaced by $K^t x = (x_1-x_2+x_3, x_3-x_4+x_5, x_2-x_3+x_4)$:
\begin{align*}
x_1 &\, = \, (x_1 - x_2 + x_3) u_1 + (x_2 - x_3 + x_4) u_1 u_2 \\
x_2 &\, = \, (x_3 - x_4 + x_5) u_2 u_3
+ (x_2 - x_3 + x_4) u_2 u_3 u_4
+ (x_2 - x_3 + x_4) u_1 u_2 \\
x_3 &\, =\, (x_1 - x_2 + x_3) u_3 u_4
+ (x_3 - x_4 + x_5) u_2 u_3
+ (x_2 - x_3 + x_4) u_2 u_3 u_4 \\
x_4 &\, = \, (x_1 - x_2 + x_3) u_3 u_4
+ (x_2 - x_3 + x_4) u_4 u_5
+ (x_2 - x_3 + x_4) u_2 u_3 u_4 \\
x_5 &\, = \, (x_3 - x_4 + x_5) u_5
+ (x_2 - x_3 + x_4) u_4 u_5.
\end{align*}
For generic choices of $x$, these equations have $2$ solutions $u(x)$ with nonzero coordinates. This matches the topological Euler characteristic of the very affine variety $U^\circ \simeq M_{0,5}$. %To see that $(-1)^2 \chi(U^\circ) = 2$, notice that $U^\circ \simeq (\mathbb{C}^*)^2 \setminus V((1+t_1)(1+t_2)(1+t_2+t_1t_2)) \simeq (\mathbb{C}^*)^2 \setminus V((u_1+u_2)(1+u_2)(1+u_1+u_2))$, where the last isomorphism is obtained by applying the isomorphism of tori $(t_1,t_2) \mapsto (t_1t_2, t_2)$. Hence, $U^\circ$
\end{example}

In summary, the scattering equations define $(-1)^d \cdot \chi(U^\circ)$ critical points of the monomial $u^x$ on $U^\circ$. Solving the scattering equations amounts to computing the fiber of the algebraic moment map $\mu_{U,K^tx}$ over $x$. This interpretation has been used in particle physics to express certain \emph{scattering amplitudes} as volumes of convex polytopes \cite[Claim 4]{ArkaniHamed2021}.

\paragraph{Acknowledgements.} 
We thank Federico Ardila Mantilla, Hadleigh Frost, Bernd Sturmfels, and Hugh Thomas for useful conversations. We thank Prashanth Raman for sharing examples of binary geometries to test our construction. The authors were partially supported by the European Union (ERC, UNIVERSE PLUS, 101118787)\footnote{\tiny Views and opinions expressed are, however, those of the author(s) only and do not necessarily reflect those of the European Union or the European Research Council Executive Agency. Neither the European Union nor the granting authority can be held responsible for them.}.

\paragraph{Statements and Declarations.} There are no competing interests to disclose.

\small
\bibliographystyle{alphaurl}
\bibliography{refs}

\begin{comment}
\noindent{\bf Authors' addresses:}
\smallskip

\noindent Veronica Calvo Cortes, MPI-MiS Leipzig
\hfill {\tt veronica.calvo@mis.mpg.de}

\noindent Simon Telen, MPI-MiS Leipzig
\hfill {\tt simon.telen@mis.mpg.de}
\end{comment}
\end{document}